\documentclass[12pt,a4paper]{article}
\usepackage{commath}
\usepackage{ragged2e}
\usepackage{graphicx}
\usepackage{enumitem}
\usepackage{tabularx}
\usepackage{rotating}
\usepackage{multirow}
\usepackage{float}
\usepackage{amsthm,amssymb,mathrsfs,amsmath}
\usepackage[bindingoffset=0.2in,left=0.5in,right=0.5in,top=1in,bottom=1in,footskip=.25in]{geometry}
\usepackage{cite}
\usepackage[colorlinks,citecolor=blue,urlcolor=blue,bookmarks=false,hypertexnames=true]{hyperref}
\hypersetup{citecolor=blue}

\newtheorem{theorem}{Theorem}[section]

\newtheorem{lemma}[theorem]{Lemma}

\title{Uniform Haar Wavelet Solutions for Fractional Regular $\beta$-Singular BVPs Modeling Human Head Heat Conduction under Febrifuge Effects}\author{Narendra Kumar$^2$\thanks{$^2$narendrakumar.knp@gmail.com}, Lok Nath Kannaujiya$^1$\thanks{$^1$loknath.kannaujia159@gmail.com},  Amit K. Verma$^3$\thanks{$^3$Corresponding author: akverma@iitp.ac.in}\\\small{\it{$^{1,3}$ Department of Mathematics,}} \\\small{\it{Indian Institute of Technology Patna,}}\\\small{\it{ Bihta, Patna 801103, (BR) India.}}\\
\small{\it{$^{2}$ Department of Mathematics,}} \\\small{\it{Indian Institute of Technology Jodhpur,}}\\\small{\it{ Karwar, Jodhpur 342030, India.}}}
\date{}
\begin{document}
\justifying
\maketitle
\begin{abstract} 
This paper introduces nonlinear fractional Lane-Emden equations of the form, $$ D^{\alpha} y(x) + \frac{\lambda}{x^\beta}~ D^{\beta} y(x) + f(y) =0, ~ ~1 < \alpha \leq 2, ~~ 0< \beta \leq 1, ~~ 0 < x < 1,$$ 
subject to boundary conditions,
$$ y'(0) =\mathbf{a} , ~~~ \mathbf{c}~ y'(1) + \mathbf{d}~ y(1) = \mathbf{b},$$
where, $D^\alpha, D^\beta$ represent Caputo fractional derivative, $\mathbf{a, b,c,d} \in \mathbb{R}$, $ \lambda = 1, 2$,  and $f(y)$ is non linear function of $y.$ We have developed  collocation method namely, uniform fractional Haar wavelet collocation method and used it to compute solutions. The proposed method combines the quasilinearization method with the Haar wavelet collocation method. In this approach, fractional Haar integrations is used to determine the linear system, which, upon solving, produces the required solution. Our findings suggest that as the values of $(\alpha, \beta)$ approach $(2,1),$ the solutions of the fractional and classical Lane-Emden become identical.
\end{abstract}
\section{Introduction}
Fractional differential equations (FrDE) have recently become a valuable tool for modeling a wide range of phenomena in applied research. Fractional derivatives are effective in capturing the memory and hereditary traits observed in several real-life problems. As the applications of FrDE continue to expand, researchers have been working on developing better techniques for solving them. Fractional derivatives can be used to represent non-linear seismic oscillations and fluid-dynamic traffic models \cite{podlubny1999introduction, he1998nonlinear, he1999some}. There are several operators used to study fractional differential equations, including the Grunwald-Letnikov fractional derivative \cite{scherer2011grunwald, miller1993introduction}, Riemann-Liouville fractional derivative (R-L FD) \cite{kilbas2006theory, miller1993introduction}, Caputo fractional derivative \cite{odibat2006approximations, podlubny1999introduction}, and Jumarie fractional derivative \cite{jumarie2006modified, jumarie2007fractional, jumarie2009table}. One-dimensional Lane-Emden equations, prevalent across various scientific and engineering disciplines, are the focus of significant attention \cite{math8071045}. Comprehensive discussions on the existence and uniqueness of solutions for such equations are excellently provided in notable works such as \cite{verma2011monotone, pandey2010solvability, pandey2010monotone}.

Let us consider the classical Lane-Emden equation,
\begin{equation}\label{3}
- \{ \theta''(\tau) + \frac{\kappa}{\tau} \theta'(\tau)\} = \psi(\tau,\theta),
\end{equation}
subject to the following boundary conditions (b.c.) $$\theta'(0) = a, ~~~\theta(1) = b,$$ where $a, b$ are constants, $\kappa =1, 2$ and $\psi(\tau,\theta)$ is the nonlinear function. These equations occur in many research areas and engineering applications. Chandrasekhar studied stellar structure \cite{chandrasekhar1939book, chandrasekhar1957introduction} and derived the following BVP,
\begin{equation}\label{chandra_DE}
     \frac{d^2 \theta(t)}{dt^2} + \frac{2}{t} \frac{d \theta(t)}{dt}+ \theta^5=0,
\end{equation}
with b.c., \begin{equation}
    \theta'(0)= 0, ~ ~  ~ ~~ \theta(1) = \frac{\sqrt{3}}{2}.
\end{equation}
Poisson-Boltzmann equation's solution for the critical condition of inflammability studied by Chambr\`e \cite{chambre1952solution},
\begin{equation}\label{chamber_DE}
     \frac{d^2 \theta}{dz^2} + \frac{k}{z} \frac{d \theta}{dz}+ \delta e^{-\theta}=0,
\end{equation}
with b.c.,
\begin{eqnarray}\begin{cases}
 \frac{d \theta}{d z} =0  & \mbox{at}~  z=0,\\
  \theta=0  & \mbox{at}~ z=1.
\end{cases}
\end{eqnarray}
In 1981 Anderson et al.\cite{anderson1981complementary}, in 2019, Zahoor et al. \cite{raja2018new} studied the steady state of spherical symmetrical heat conduction problem,
\begin{equation}\label{8}
    \frac{d^2 P}{dt^2} + \frac{2}{t} \frac{d P}{dt}+ \frac{\delta}{k} e^{-\alpha P}=0,
\end{equation}
with b.c.,
\begin{eqnarray}\begin{cases}\label{11}
  \frac{d P}{d t} =0, & \mbox{at} ~ t=0,\\
  -k \frac{d P}{d t} = \beta (T-T_{\alpha}), & \mbox{at}~ t=1.
\end{cases}
\end{eqnarray}
The given equation accurately describes the heat conduction issues that affect the human head. However, condition \eqref{11} becomes irrelevant when a febrifuge, such as ibuprofen or acetaminophen, is used.
Kang-Jia Wang \cite{wang2020new} modified the heat conduction equation taking the effect of febrifuge and proposed the fractional version of the heat conduction problem \eqref{8},
  \begin{equation}
      \frac{d^{2\tau}{u}}{dx^{2\tau}} + \frac{2}{x} \frac{d^{\tau}u}{dx^{\tau}} + \lambda e^{-m u} = 0,
  \end{equation}
  with b.c., \begin{equation}
      \begin{cases} \frac{du}{dx} = 0, & \mbox{at}~x = 0,\\
 \frac{du}{dx} = k(1-u),  & \mbox{at}~ x = 1,
\end{cases}
  \end{equation}
where $0< \tau \leq 1$ and $k$ is Biot Number. In this paper, we further explore and propose a new class of fractional Lane-Emden equations that arise in various branches of science and engineering in their classical form. 
As far as fractional Lane-Emden equations are considered, hardly any investigations are observed (Wang \cite{wang2020new}). 

In recent years, researchers have increasingly turned to wavelet methods for tackling the computation of solutions to differential equations. Among the myriad of wavelet types, Haar wavelets stand out for their distinct properties such as orthogonality and compact support, etc. Researchers have proposed various methods using Haar wavelets to solve differential equations of higher order and have tested these methods on different test problems. The results reveal that, despite employing a limited number of grid points, the Haar wavelet method exhibits notable accuracy. For instance, Majak et al. \cite{majak2015convergence} utilized the Richardson extrapolation technique to explore the convergence of Haar wavelets in solving nth-order differential equations pertaining to functionally graded beams. 

However, research on singular differential equations is scarce, and there are many possibilities for investigation. Therefore, we focus on the fractional version of the Lane-Emden equations. In this paper, we propose the following class of nonlinear fractional Lane-Emden equation,
\begin{equation}\label{P3_1}
D^\alpha y(x) + \frac{\lambda}{x^\beta} D^\beta y(x) + f(y) =0, ~~~ 0 < x <1,~ 0< \beta \leq 1, ~ 1 < \alpha \leq 2,
\end{equation}
subject to boundary condition (b.c.), $$  y'(0) = \mathbf{a} ~~~~ \mathbf{c}~ y'(1) + \mathbf{d}~ y(1) = \mathbf{b}, $$
 where $\mathbf{a, b,c,d} \in \mathbb{R},$ $D^\alpha, D^\beta$ represent Caputo fractional derivative,  $\lambda=1,2$ and $f(y)$ is a nonlinear function of $y$. Here $x=0$ is regular $\beta$-singular point of equation \eqref{P3_1} (see \cite{kilbas2006theory}).

It has been observed that there are only a few investigations on fractional Lane-Emden equations occur \cite{wang2020new, saeed2017haar}. This paper fills this gap by introducing a novel class of fractional Lane-Emden equation \eqref{P3_1}. We propose a new approach and derive a method using quasilinearization and the Haar wavelet method, called the uniform fractional Haar wavelet collocation method (UFHWCM) to simulate solutions for \eqref{P3_1}. Leveraging fractional integrals of Haar wavelets, the proposed method transforms the problem into a system of algebraic equations, effectively capturing solutions. A comprehensive analysis is also provided to examine the stability and convergence characteristics of UFHWCM. We conduct experiments on three test cases and observe that when $\alpha$ approaches $2$ and $\beta$ approaches $1$, the solutions obtained using UFHWCM are arbitrarily close to the exact solution of the classical Lane-Emden equation for $\alpha=2$ and $\beta=1$. The proposed approach is novel and may be studied further in various directions. We present the flowchart depicted in Figure \ref{Flowchart}, which outlines the overall arrangement of the proposed work.

Following this introduction, the paper is structured as follows: Section \ref{prelim} introduces key definitions and concepts. Section \ref{P3_Method} delves into the analysis of the collocation method UFHWCM when applied to fractional Lane-Emden equations under specific boundary conditions. Section \ref{convergence} presents its convergence and stability analysis. In Section \ref{P3_examples}, we examine three numerical experiments in detail. Finally, Section \ref{P3_conclusion} summarizes the outcomes and implications of our proposed method.

\section{Preliminary}\label{prelim}
This section provides an overview of definitions regarding Haar wavelets and their fractional integral. The basic idea is that wavelets give rise to the basis of $L^2(\mathbb{R})$ functions and any function $f\in L^2 (\mathbb{R})$ can be expanded as a series. Since Haar wavelets are piecewise-defined, it is easy to compute the coefficients in the series expansion but difficult to program. The calculations in this paper are done by using Mathematica. 
\subsection{Uniform Haar wavelets}\label{P3_2.1_uniform}
In 2014, Lepik and Hein \cite{lepik2014haar} employed an equal step size of $\Delta t = \frac{1}{2M}$ to partition the interval $[0,1]$ into $2M$ sub-intervals. The mother wavelet function of the Haar wavelet, denoted as $h_\varrho (t)$ for $\varrho > 1$, is defined as follows:
\begin{equation}
h_\varrho (t)=
 \begin{cases}  1, & \Upsilon_1(\varrho) \leq t < \Upsilon_2(\varrho),\\
-1, & \Upsilon_2(\varrho) \leq t < \Upsilon_3(\varrho),\\
0, & otherwise,
\end{cases}
\end{equation}
such that the values $\Upsilon_1(\varrho), \Upsilon_2(\varrho),$ and $\Upsilon_3(\varrho)$ are given by,
\begin{eqnarray}
&&\Upsilon_1(\varrho) = 2k \left(\frac{M}{m}\right) \Delta t, \\
&&\Upsilon_2(\varrho) = (2k+1) \left(\frac{M}{m}\right) \Delta t, \\
&&\Upsilon_3(\varrho) = 2(k+1) \left(\frac{M}{m}\right) \Delta t,
\end{eqnarray}
where $\varrho= m+k + 1$, and the parameters $J$, $M$, $j$, $m$, and $k$ are defined such that $J$ represents the maximum level of resolution, $M = 2^J$, $j =0,1,\cdots,J$, $m = 2^j$, and $k = 0,1, \cdots, m-1$. For $\varrho=1$, the Haar wavelet function $h_1(t)$ is defined as,
 \begin{equation}
 h_1(t) = \begin{cases}  1, & 0 \leq t <1\\
 0, &  otherwise.\end{cases}
 \end{equation}
The integral of the Haar function, $\mathcal{P}_{v,\varrho} (t)$ 
is defined as, for $\varrho >1,$
\begin{equation}
\mathcal{P}_{v,\varrho} (t) = \begin{cases} 
0, & t < \Upsilon_1(\varrho),\\
\frac{1}{v!}[ t- \Upsilon_1(\varrho)]^v,  & \Upsilon_1(\varrho) \leq t \leq \Upsilon_2(\varrho),\\
\frac{1}{v!} \{ [ t- \Upsilon_1(\varrho)]^v - 2 [ t- \Upsilon_2(\varrho)]^v \}, ~& \Upsilon_2(\varrho) \leq t \leq \Upsilon_3(\varrho),\\
\frac{1}{v!} \{ [ t- \Upsilon_1(\varrho)]^v -2 [ t- \Upsilon_2(\varrho)]^v  + [ t- \Upsilon_3(\varrho)]^v \}, & t > \Upsilon_3(\varrho).
\end{cases}
\end{equation}
For $\varrho=1,$ we have $\Upsilon_1(1) =0, ~~ \Upsilon_2(1) = \Upsilon_3(1)=1,$ and the expression for $\mathcal{P}_{v,1}$ becomes 
$$\mathcal{P}_{v,1} = \frac{t^v}{v!}.$$
Collocation points are defined as follows:
\begin{eqnarray}\label{collocation}
&&\eta_{cl} = \frac{\tilde{\eta}_{cl-1} + \tilde{\eta}_{cl}}{2},~~cl= 1,2,\cdots,2M,
\end{eqnarray}
where $\tilde{\eta}_{cl}$ are the grid points and defined by the following equation,
\begin{eqnarray*}
&& \tilde{\eta}_{cl} = cl\Delta t, ~~cl= 0,1,\cdots,2M.
\end{eqnarray*}
\subsection{Some Definitions} Several definitions of fractional derivatives are available. In this work, we use the following two versions.\\
\textbf{Riemann-Liouville Fractional Integral \cite{kilbas2006theory, podlubny1999introduction, miller1993introduction}:} The Riemann-Liouville fractional integral of a function $\psi (t)$ is given as,
\begin{equation}
I^\nu \psi (t) = \frac{1}{\Gamma{(\nu)}} \int_0^t (t-\tau)^{\nu -1} \psi (\tau) d\tau, ~~~ (m-1) < \nu \leq m, ~~~m = \lceil \nu \rceil.
\end{equation}
\textbf{Caputo Fractional Derivative \cite{podlubny1999introduction, odibat2006approximations}:} The Caputo fractional derivative of order $\nu > 0, $ of a function $\phi(t)$ is defined as, 
         \begin{equation}
             D^{\nu} \phi(t) = \frac{1}{\Gamma(m-\nu)} \int_{0}^{t} (t-\tau)^{m-\nu -1} \phi^m (\tau) d\tau,~~ (m-1) < \nu \leq m, ~~~ m = \lceil \nu \rceil.
         \end{equation}
In what follows we recall some important Lemmas from \cite{podlubny1999introduction}, that will be used in this paper.
\begin{lemma}\label{P3_lemma2_1}
 Let $\gamma >0$ and $ \psi(\tau) = \tau^\nu,$ then, $$ I^\gamma \psi(\tau) = \frac{\Gamma(\nu +1)}{\Gamma (\nu + \gamma + 1)} \tau^{\gamma + \nu}.$$
\end{lemma}
\begin{lemma}\label{P3_lemma2_2}
If $ \psi(\tau) = c,$ where $c$ is constant and $\gamma >0,$ then, $D^\gamma \psi(\tau) =0.$
\end{lemma}
\begin{lemma}\label{P3_lemma2_3}
Let $\gamma >0, n = \lceil \gamma \rceil$ and $ \phi \in AC^n[a,b],$ then,
$$ I^\gamma~D^\gamma \phi(\zeta)= \phi(\zeta) - \sum_{k=0}^{n-1} \frac{\zeta^k}{k!} \phi^{(k)} (0).$$
\end{lemma}
\subsection{Fractional Integration of Uniform Haar Wavelets}
In the study by Saeed et al. (2017) \cite{saeed2017haar}, the Riemann-Liouville fractional integration of order $\upsilon$ for the uniform Haar wavelets is expressed as follows, when $\varrho = 1$:
\begin{equation}\label{P3_2_2}
I^\upsilon h_1 (t) = \frac{t^\upsilon}{\Gamma(\upsilon +1)},
\end{equation}
and, for $\varrho>1,$ 
\begin{align}\label{P3_2_3}
 p_{\upsilon, \varrho} (t) = I^\upsilon h_\varrho(t) &= \frac{1}{\Gamma(\upsilon)} \int_0^t (t-s)^{\upsilon-1} h_\varrho (s) ds,\nonumber\\
 &= \frac{1}{\Gamma(\upsilon +1)} \begin{cases}
(t-\Upsilon_1(\varrho))^\upsilon & \Upsilon_1(\varrho) \leq t < \Upsilon_2(\varrho), \\
(t-\Upsilon_1(\varrho))^\upsilon -2 (t-\Upsilon_2(\varrho))^\upsilon , & \Upsilon_2(\varrho) \leq t < \Upsilon_3(\varrho),\\
(t-\Upsilon_1(\varrho))^\upsilon -2 (t-\Upsilon_2(\varrho))^\upsilon + (t-\Upsilon_3(\varrho))^\upsilon, & t \geq \Upsilon_3(\varrho),
\end{cases}
\end{align}
where $\Upsilon_1(\varrho) ~ \Upsilon_2(\varrho), \Upsilon_3(\varrho)$ are defined as in equation \ref{P3_2.1_uniform}. The Haar matrix $H$ is constructed using collocation points \eqref{collocation} where $H(\varrho,cl) = h_\varrho (\eta_{cl} ).$ Furthermore, by replacing the collocation points \eqref{collocation} into equations \eqref{P3_2_2} and \eqref{P3_2_3}, we derive the integration matrix $P$ for fractional orders of the Haar function, resulting in $P_{\upsilon}(\varrho,cl) = p_{\upsilon, \varrho}(\eta_{cl}),$
$$ P = 
\begin{bmatrix}
p_{\upsilon, 1}(\eta_{cl}(1)) & p_{\upsilon,1}(\eta_{cl}(2)) \cdots & p_{\upsilon, 1}(\eta_{cl}(2M))\\
p_{\upsilon, 2}(\eta_{cl}(1)) & p_{\upsilon, 2}(\eta_{cl}(2)) \cdots & p_{\upsilon, 2}(\eta_{cl}(2M))\\
\cdots \\
\cdots \\
p_{\upsilon, 2M}(\eta_{cl}(1)) & p_{\upsilon, 2M}(\eta_{cl}(2)) \cdots &p_{\upsilon, 2M}(\eta_{cl}(2M))
\end{bmatrix}
.$$

\begin{figure}[H] 
  \centering
  \includegraphics[width=1.0\textwidth]{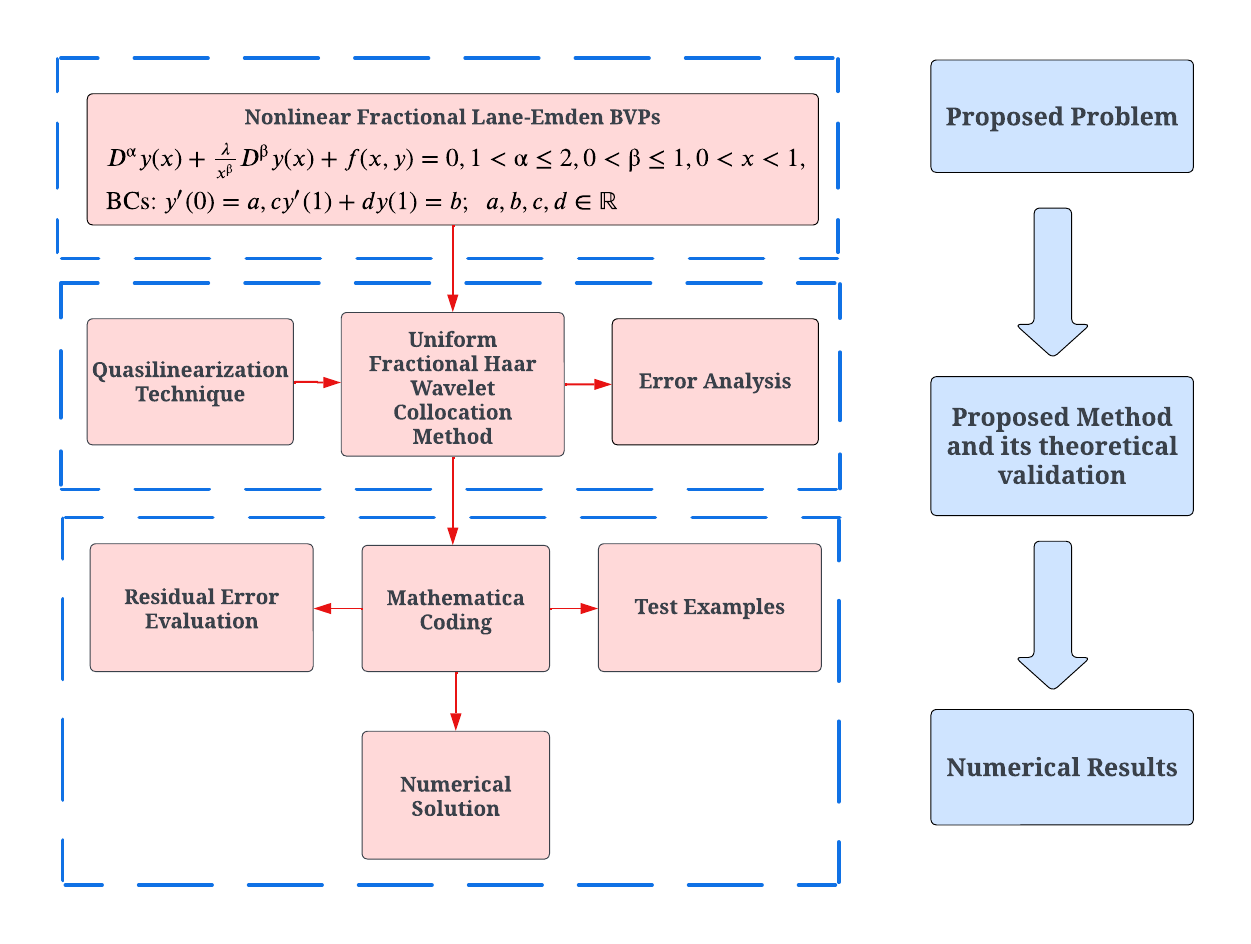} 
  \caption{The demonstration of the method design of the proposed work.} 
  \label{Flowchart}
\end{figure}

 \section{Proposed Method \& Its Algorithm}\label{P3_Method}
In this section, we propose method namely the uniform fractional Haar wavelet collocation method (UFHWCM). We use the collocation approach along with quasilinearization and the Haar wavelet method. The flowchart depicted in Figure \ref{Flowchart} elucidates the methodological design of the proposed study.

Quasilinearization is a technique that transforms a nonlinear differential equation into a sequence of linear differential equations. The solution is then obtained by taking the limit of the sequence of solutions of these linear differential equations. This method was first introduced by Bellman and Kabala \cite{bellman1965quasilinearization}. For any nonlinear ordinary differential equation, quasilinearization can be applied, and is explained briefly in \cite{mandelzweig2001quasilinearization}.\\
Now, using quasilinearization on \eqref{P3_1}, we get,
\begin{equation}\label{P3_3_1}
 D^\alpha y_{r+1} (x) + \frac{\lambda}{x^\beta} D^\beta y_{r+1} (x)= - \Big[ f(y_{r}(x)) + (y_{r+1}(x) - y_r(x)) f_{y_r}(y_{r}(x))\Big],
\end{equation}
with b.c., \begin{equation}\label{P3_3_1_bc}
 y_{r+1}'(0)=\mathbf{a}, ~~~ \mathbf{c}~ y_{r+1}'(1) + \mathbf{d}~ y_{r+1}(1) = \mathbf{b}.
\end{equation}
Let us consider, 
\begin{equation}\label{P3_3_2}
D^\alpha y_{r+1}(x) = \sum_{\varrho=1}^{2M} a_\varrho h_\varrho(x).
\end{equation}
The lower order derivative can be obtained by taking the fractional integral of equation \eqref{P3_3_2}, i.e., 
\begin{equation}\label{P3_3_3}
I^\alpha D^\alpha y_{r+1}(x) = \sum_{\varrho=1}^{2M} a_\varrho I^\alpha h_\varrho(x),
\end{equation}
since $1 <\alpha \leq 2,$ so, $n = \lceil \alpha \rceil =2,$ so, using Lemma \ref{P3_lemma2_3}, we get,
\begin{equation}\label{P3_3_3_1}
I^\alpha D^\alpha y_{r+1}(x) = y_{r+1}(x) - \sum_{k=0}^{n-1} \frac{x^k}{\Gamma(k+1)} y_{r+1}^{(k)}(0) = y_{r+1}(x) - y_{r+1}(0) - x y_{r+1}'(0),
\end{equation}
using equation \eqref{P3_3_1_bc} and \eqref{P3_3_3_1}, equation \eqref{P3_3_3} reduces to,
\begin{equation}\label{P3_3_4}
 y_{r+1}(x) = \sum_{\varrho=1}^{2M} a_\varrho I^\alpha h_\varrho(x) +  \mathbf{a} x + y_{r+1}(0).
\end{equation}
Further, we obtain,
\begin{equation*}
y_{r+1}'(x) = \sum_{\varrho=1}^{2M} a_\varrho I^{\alpha-1} h_\varrho(x) +  \mathbf{a}.
\end{equation*}
At $x=1,$ we have,
\begin{align*}
y_{r+1}(1) &= \sum_{\varrho=1}^{2M} a_\varrho I^\alpha h_\varrho(1) + \mathbf{a}+ y_{r+1}(0),\\
 y_{r+1}'(1) &= \sum_{\varrho=1}^{2M} a_\varrho I^{\alpha-1} h_\varrho(1) +  \mathbf{a}.
\end{align*}
using b.c. \eqref{P3_3_1_bc}, we get,
\begin{align*}
y_{r+1}(0) &= - \frac{\mathbf{c}}{\mathbf{d}}\sum_{\varrho=1}^{2M} a_\varrho I^{\alpha-1} h_\varrho(1) -\sum_{\varrho=1}^{2M} a_\varrho I^\alpha h_\varrho(1) + \Big(\frac{\mathbf{b} -\mathbf{a} \mathbf{c}- \mathbf{a}\mathbf{d}}{\mathbf{d}}\Big).
\end{align*}
So, equation \eqref{P3_3_4} reduces to,
\begin{equation}\label{P3_3_5}
 y_{r+1}(x) = \sum_{\varrho=1}^{2M} a_\varrho I^\alpha h_\varrho(x) + \mathbf{a} x - \frac{\mathbf{c}}{\mathbf{d}}\sum_{\varrho=1}^{2M} a_\varrho I^{\alpha-1} h_\varrho(1) -\sum_{\varrho=1}^{2M} a_\varrho I^\alpha h_\varrho(1) + \Big(\frac{\mathbf{b} -\mathbf{a} \mathbf{c}- \mathbf{a}\mathbf{d}}{\mathbf{d}}\Big).
\end{equation}
Applying Caputo fractional derivative $D^\beta$ on both side in equation \eqref{P3_3_5}, and using Lemma \ref{P3_lemma2_2}, we get,
\begin{equation}\label{P3_3_6}
D^\beta y_{r+1}(x)= \sum_{\varrho=1}^{2M} a_\varrho I^{\alpha - \beta} h_\varrho(x) + \frac{\mathbf{a}~ \Gamma{(2)}}{\Gamma{(2-\beta)}} x^{1-\beta}.
\end{equation}
Now substituting equations \eqref{P3_3_2}, \eqref{P3_3_5} and \eqref{P3_3_6}, in equation \eqref{P3_3_1} and using collocation points \eqref{collocation}, we get a system of linear algebraic equation of the following form,
\begin{eqnarray}
&& \psi_1 (a_1, a_2, \cdots, a_{2M})=0, \nonumber\\
 &&\psi_2 (a_1, a_2, \cdots, a_{2M})=0, \nonumber\\
 && \vdots \nonumber \\
 && \psi_{2M} (a_1, a_2, \cdots, a_{2M})=0.
\end{eqnarray} 
This system of equations can be conveniently represented in matrix-vector form as follows,
$$ TV = B,$$
where $T$ denotes a square matrix of dimension $2M \times 2M$, $B$ represents a vector comprising constants, and $V =(a_1, a_2, \cdots, a_{2M})$ is the vector of wavelet coefficients. Solving the aforementioned system of equations yields the wavelet coefficients $a_\varrho$. Subsequently, we substitute these determined coefficients into equation \eqref{P3_3_5} to derive the desired solution.
\subsection{Algorithm for UFHWCM}
Here, we outline the steps for our method to solve fractional Lane-Emden equations \eqref{P3_1}. The procedure involves the following key steps;\\
\textit{Step 1(Quasilinearization):} Transform the nonlinear Lane-Emden equations (\ref{P3_1}) into a sequence of linear Lane-Emden equations (\ref{P3_3_1}) using the quasilinearization scheme.\\
\textit{Step 2(Haar wavelet collocation):} Utilize the Haar wavelet collocation method, incorporating the boundary conditions into the sequence of linear Lane-Emden equations.\\
\textit{Step 3(Substitution and replacement):} We insert equations (\ref{P3_3_2}), (\ref{P3_3_5}), and (\ref{P3_3_6}) into equation (\ref{P3_3_1}), and we replace the variable $x$ with the collocation points defined in (\ref{collocation}).\\
\textit{Step 4(Solve linear system):} Following step $3$, we obtain a system of linear algebraic equations. These equations can be solved by initiating with an initial guess $x_0$ to determine the wavelet coefficients.\\
\textit{Step 5(First approximate solution):} We substitute the obtained wavelet coefficients into \eqref{P3_3_5} to get our first approximate solution.\\
\textit{Step 6(Iterative refinement):} Continue steps $2-5$ for $r$ iterations to get the Haar solutions until the desired accuracy.

\section{Convergence \& Stability Analysis} \label{convergence}

In this section, we present the convergence and stability analysis of the proposed method. 

\subsection{Convergence Analysis}
To prove the convergence, we require the following result.
\begin{lemma}\label{P3_lemma4_2} 
Let $r,s >0$ and $1 < \alpha \leq 2,$ then,
$$r^\alpha - s^\alpha \leq (r-s)^\alpha.$$
\end{lemma}
\begin{proof} 
Let $\frac{s}{r} = x,$ then \begin{equation}\label{P3_lemma4_2_1}
(r-s)^\alpha + s^\alpha = r^\alpha \Big[ \Big(1-\frac{s}{r}\Big)^\alpha + \Big(\frac{s}{r}\Big)^\alpha \Big]=r^\alpha f(x).
\end{equation}
Since, $f(x)$ is continuous and for $1 < \alpha \leq 2,$ then, $\min_{x\in[0,1]}f(x)=1$.
Hence 
\begin{equation*}
r^\alpha \Big[ \Big(1-\frac{s}{r}\Big)^\alpha + \Big(\frac{s}{r}\Big)^\alpha \Big] \geq  r^\alpha  \min_{x\in[0,1]}f(x)  \geq r^\alpha \times 1 = r^\alpha.\\
\end{equation*}
Thus, $$r^\alpha - s^\alpha \leq (r-s)^\alpha.$$
This completes the proof.
\end{proof} 
\begin{theorem}
Let $E_M =  y_E(x) - y_M(x)$ where $y_E(x) ~\&~ y_M(x)$ is exact and Haar  solution respectively,  and $D^{\alpha +1}y(x)$ be continuous on $[0,1]$. Further, let there exist $\epsilon>0$ such that $|D^{\alpha +1} y(x)| \leq \epsilon$ on $[0,1]$ and $k$ is any positive real value then, 
$$\|E_M\|_2 \leq k \epsilon \Big( \frac{1}{2^{J+1}}\Big)^{\alpha-1} \Big\{ \Big( \frac{1}{2^{J+1}}\Big) \Big(\frac{1}{2^\alpha -1} \Big) + \frac{\mathbf{c}}{2\mathbf{d}} \Big(\frac{1}{2^{\alpha -1} -1}\Big) \Big\}.$$
\end{theorem}

\begin{proof} 
Since, we have,
\begin{align*}
 D^\alpha y(x) &= \sum_{\varrho=1}^{2M} a_\varrho h_\varrho(x),
\end{align*}
exact solution,
\begin{align*}
y_E(x) &= \sum_{\varrho=1}^{\infty} a_\varrho I^\alpha h_\varrho(x) + \mathbf{a} x - \frac{\mathbf{c}}{\mathbf{d}}\sum_{\varrho=1}^{\infty} a_\varrho I^{\alpha-1} h_\varrho(1) -\sum_{\varrho=1}^{\infty} a_\varrho I^\alpha h_\varrho(1) \\
 & \quad \quad+ \Big(\frac{\mathbf{b} -\mathbf{a} \mathbf{c}- \mathbf{a}\mathbf{d}}{\mathbf{d}}\Big),
\end{align*}
and, Haar solution,
\begin{align*}
y_M(x) &= \sum_{\varrho=1}^{2M} a_\varrho I^\alpha h_\varrho(x) + \mathbf{a} x - \frac{\mathbf{c}}{\mathbf{d}}\sum_{\varrho=1}^{2M} a_\varrho I^{\alpha-1} h_\varrho(1) -\sum_{\varrho=1}^{2M} a_\varrho I^\alpha h_\varrho(1) \\
& \quad \quad + \Big(\frac{\mathbf{b} -\mathbf{a} \mathbf{c}- \mathbf{a}\mathbf{d}}{\mathbf{d}}\Big).
\end{align*}
So, the error,
\begin{align*}
E_M  &= y_E(x) - y_M(x) = \sum_{\varrho= 2M+1}^{\infty} a_\varrho \Big(I^\alpha h_\varrho(x) - \frac{\mathbf{c}}{\mathbf{d}} I^{\alpha-1} h_\varrho(1) -I^\alpha h_\varrho(1)\Big)y.
\end{align*} 
Since, $ \varrho = 2^j +k+1$, then,
\begin{align*}
E_M = \sum_{j=J+1}^{\infty} \sum_{k=0}^{2^j-1} a_{ 2^j +k+1} \Big(I^\alpha h_{2^j +k+1}(x) - \frac{\mathbf{c}}{\mathbf{d}} I^{\alpha-1} h_{2^j +k+1}(1)-I^\alpha h_{2^j +k+1}(1)\Big).
\end{align*} 
Now, 
\footnotesize{\begin{align}\label{P3_3_1_1}
\|E_M\|_2^2 &= \int_0^1 \Big(\sum_{j=J+1}^{\infty} \sum_{k=0}^{2^j-1} a_{ 2^j +k+1} \Big(I^\alpha h_{2^j +k+1}(x)- \frac{\mathbf{c}}{\mathbf{d}} I^{\alpha-1} h_{2^j +k+1}(1)-I^\alpha h_{2^j +k+1}(1)\Big)\Big)^2 dx,\nonumber\\
&= \sum_{j=J+1}^{\infty} \sum_{k=0}^{2^j-1} \sum_{s=J+1}^{\infty} \sum_{t=0}^{2^s-1} a_{ 2^j +k+1} a_{ 2^s +t+1} \int_0^1  \Big(I^\alpha h_{2^j +k+1}(x)- \frac{\mathbf{c}}{\mathbf{d}} I^{\alpha-1} h_{2^j +k+1}(1)-I^\alpha h_{2^j +k+1}(1)\Big) \nonumber\\
 & \hspace{7.5cm} \Big(I^\alpha h_{2^s +t+1}(x)- \frac{\mathbf{c}}{\mathbf{d}} I^{\alpha-1} h_{2^s +t+1}(1)-I^\alpha h_{2^s +t+1}(1)\Big) dx, \nonumber\\
&= \sum_{j=J+1}^{\infty} \sum_{k=0}^{2^j-1} \sum_{s=J+1}^{\infty} \sum_{t=0}^{2^s-1} a_{ 2^j +k+1} a_{ 2^s +t+1} \int_0^1  \Big(p_{\alpha, 2^j +k+1}(x)- \frac{\mathbf{c}}{\mathbf{d}}p_{\alpha-1,2^j +k+1}(1)-p_{\alpha,2^j +k+1}(1) \Big) \nonumber\\
 & \hspace{7.5cm}\Big(p_{\alpha, 2^s +t+1}(x)- - \frac{\mathbf{c}}{\mathbf{d}}p_{\alpha-1, 2^s +t+1}(1)-p_{\alpha, 2^s +t+1}(1) \Big) dx.
\end{align}}
Now since, $D^\alpha y(x) = \sum_{\varrho=1}^{\infty} a_\varrho h_\varrho(x)$ then, by using the definition of uniform Haar function, we can evaluate coefficients $a_\varrho$ as,
\begin{align*}
a_\varrho &= 2^j \int_0^1 D^\alpha y(x) h_\varrho (x) dx,\\
a_\varrho &= 2^j  \Big(\int_{\Upsilon_1(\varrho)}^{\Upsilon_2(\varrho)} D^\alpha y(x) dx - \int_{\Upsilon_2(\varrho)}^{\Upsilon_3(\varrho)} D^\alpha y(x) dx \Big).
\end{align*}
Since, $D^\alpha y(x)$ be continuous on $[0,1],$ so we have,
$$a_\varrho = 2^j \{ (\Upsilon_2(\varrho)-\Upsilon_1(\varrho)) D^\alpha y (\xi_1)-(\Upsilon_3(\varrho)-\Upsilon_2(\varrho)) D^\alpha y (\xi_2)        \},$$
where $\xi_1 \in (\Upsilon_1(\varrho),\Upsilon_2(\varrho))~ \& ~~\xi_2 \in (\Upsilon_2(\varrho), \Upsilon_3(\varrho)),$ then,
$\Upsilon_2(\varrho) - \Upsilon_1(\varrho) = \Upsilon_3(\varrho) - \Upsilon_2(\varrho) = \frac{1}{2^{j+1}}.$
So we get, 
\begin{align*}
a_\varrho &= 2^j \Big\{ \frac{1}{2^{j+1}} D^\alpha y(\xi_1) - \frac{1}{2^{j+1}} D^\alpha y(\xi_2) \Big\},\\
&= \frac{1}{2} \{D^\alpha y(\xi_1) - D^\alpha y(\xi_2)\},\\
&= \frac{1}{2} (\xi_1 - \xi_2) D^{\alpha+1} y(\xi), ~~\xi \in (\xi_1, \xi_2).
\end{align*}
Since, $D^{\alpha +1} y(x)$ is bounded such that $|D^{\alpha+1} y(x)| \leq \epsilon,$ which implies, $a_\varrho \leq  \epsilon ~\Big(\frac{1}{2^{j+1}}\Big).$ Since, the fractional integral of uniform Haar wavelet is given in \eqref{P3_2_2} and \eqref{P3_2_3}, we first derive the bounds in each of the intervals, $[\Upsilon_1(\varrho), \Upsilon_2(\varrho)),~~ [\Upsilon_2(\varrho), \Upsilon_3(\varrho))$ ~~$\& ~~\Upsilon_3(\varrho) \leq x <1.$  Now for, $x \in [\Upsilon_1(\varrho), \Upsilon_2(\varrho))$, i.e., $ \Upsilon_1(\varrho) \leq x < \Upsilon_2(\varrho),$ we have, $$p_{\alpha, \varrho} (x) = \frac{(x-\Upsilon_1(\varrho))^\alpha}{\Gamma(\alpha +1)}.$$ Here, since, $p_{\alpha, \varrho}(x)$ is monotonically increasing. So,
\begin{align*}
0 \leq  p_{\alpha, \varrho}(x) &\leq p_{\alpha, \varrho}(\Upsilon_2(\varrho)),\\
|p_{\alpha, \varrho}(x)| &\leq \frac{(\Upsilon_2(\varrho)-\Upsilon_1(\varrho))^\alpha}{\Gamma(\alpha +1)},\\
|p_{\alpha, \varrho}(x)| &\leq \frac{1}{\Gamma(\alpha +1)} \Big( \frac{1}{2^{j +1}}\Big)^\alpha.
\end{align*}
Now, since $p_{\alpha, \varrho}(x)$  is continuous on compact interval $[\Upsilon_2(\varrho), \Upsilon_3(\varrho)],$ so, the function $p_{\alpha, \varrho}(x)$ attains it's global maximum in this interval. Since,
\begin{align*}
p_{\alpha, \varrho}(x) &= \frac{1}{\Gamma(\alpha +1)} \{ (x-\Upsilon_1(\varrho))^\alpha - 2 (x-\Upsilon_2(\varrho))^\alpha \}, ~~~~~~~~ \Upsilon_2(\varrho) \leq x \leq \Upsilon_3(\varrho).\\
\end{align*}
Then,
\begin{align*}
p'_{\alpha, \varrho} (x) &=  \frac{1}{\Gamma(\alpha +1)} \{ \alpha (x-\Upsilon_1(\varrho))^{\alpha-1} - 2 \alpha(x-\Upsilon_2(\varrho))^{\alpha-1} \},\\
&=  \frac{1}{\Gamma(\alpha)} \{  (x-\Upsilon_1(\varrho))^{\alpha-1} - 2 (x-\Upsilon_2(\varrho))^{\alpha-1} \}.
\end{align*}
So, $p_{\alpha, \varrho} (x)$ will be monotonically increasing if
\begin{equation*}
p'_{\alpha, \varrho} (x) \geq 0,
\end{equation*}
after simplification, we get, $$x \leq \frac{2^{\frac{1}{\alpha -1}} \Upsilon_2(\varrho) -\Upsilon_1(\varrho)}{2^{\frac{1}{\alpha -1}}-1}.$$\\
Thus $p_{\alpha, \varrho} (x)$ is monotonically increasing if $x \leq \frac{2^{\frac{1}{\alpha -1}} \Upsilon_2(\varrho) -\Upsilon_1(\varrho)}{2^{\frac{1}{\alpha -1}}-1} $ in $[\Upsilon_2(\varrho), \Upsilon_3(\varrho)].$ By using the method of contradiction, we can easily prove that, $$\Upsilon_3(\varrho) \leq \frac{2^{\frac{1}{\alpha -1}} \Upsilon_2(\varrho) -\Upsilon_1(\varrho)}{2^{\frac{1}{\alpha -1}}-1} .$$
Therefore, we have, $$0 \leq p_{\alpha,\varrho} (x) \leq p_{\alpha,\varrho} (\Upsilon_3(\varrho)), ~~~~~~\forall x \in [\Upsilon_2(\varrho), \Upsilon_3(\varrho)],$$
and, \begin{align*}
|p_{\alpha, \varrho}(x)| &\leq \frac{1}{\Gamma(\alpha +1)} \{ (\Upsilon_3(\varrho)-\Upsilon_1(\varrho))^\alpha - 2 (\Upsilon_3(\varrho) -\Upsilon_2(\varrho))^\alpha\},\\
&= \Big(\frac{2^\alpha -2}{\Gamma(\alpha +1)}\Big) \Big(\frac{1}{2^{j+1}}\Big)^\alpha, ~~~~~~\forall x \in [\Upsilon_2(\varrho), \Upsilon_3(\varrho)].
\end{align*}
For $ 1> x \geq \Upsilon_3(\varrho),$ we have,
$$ p_{\alpha, \varrho} (x) = \frac{1}{\Gamma{(\alpha +1)}} \Big[ (x-\Upsilon_1(\varrho))^\alpha - 2 (x-\Upsilon_2(\varrho))^\alpha + (x-\Upsilon_3(\varrho))^\alpha \Big].$$
By using Lemma \ref{P3_lemma4_2}, we get,
\begin{align*}
p_{\alpha, \varrho} (x) &=  \frac{1}{\Gamma{(\alpha +1)}} \Big[ (x-\Upsilon_1(\varrho))^\alpha - 2 (x-\Upsilon_2(\varrho))^\alpha + (x-\Upsilon_3(\varrho))^\alpha \Big],\\
&=  \frac{1}{\Gamma{(\alpha +1)}} \Big[ (x-\Upsilon_1(\varrho))^\alpha -  (x-\Upsilon_2(\varrho))^\alpha + (x-\Upsilon_3(\varrho))^\alpha -(x-\Upsilon_2(\varrho))^\alpha \Big],\\
&\leq \frac{1}{\Gamma{(\alpha +1)}} \Big[ \{\Upsilon_2(\varrho) -\Upsilon_1(\varrho)\}^\alpha + \{\Upsilon_2(\varrho)-\Upsilon_3(\varrho)\}^\alpha\big],\\
&\leq \frac{1}{\Gamma{(\alpha +1)}} \Big[ \Big( \frac{1}{2^{j+1}}\Big)^\alpha + \Big( \frac{-1}{2^{j+1}}\Big)^\alpha\Big],\\
&= \frac{1}{\Gamma{(\alpha +1)}} (1 + (-1)^\alpha )\Big( \frac{1}{2^{j+1}}\Big)^\alpha.
\end{align*}
Hence, we have, \begin{equation*}
| p_{\alpha, \varrho} (x)| \leq k \Big( \frac{1}{2^{j+1}}\Big)^\alpha, ~~\forall x \in [0,1].
\end{equation*}
From the above, we deduce the following, 
\begin{equation*} 
| p_{\alpha-1, \varrho} (x)| \leq k \Big( \frac{1}{2^{j+1}}\Big)^{\alpha-1}, ~~\forall x \in [0,1].
\end{equation*}
Then, \begin{eqnarray}
&& |p_{\alpha, \varrho}(1)|\leq k \Big( \frac{1}{2^{j+1}}\Big)^\alpha. \nonumber\\
&& |p_{\alpha-1, \varrho}(1)|\leq k \Big( \frac{1}{2^{j+1}}\Big)^{\alpha-1} \nonumber.
\end{eqnarray}
Now, using inequality, \begin{equation}
\Big(p_{\alpha, \varrho}(x)-\frac{\mathbf{c}}{\mathbf{d}} p_{\alpha-1, \varrho}(1)-p_{\alpha, \varrho}(1) \Big) \leq ~ |p_{\alpha, \varrho}(x)| + \frac{\mathbf{c}}{\mathbf{d}} |p_{\alpha-1, \varrho}(1)|+|p_{\alpha, \varrho}(1)|,
\end{equation}
we get, \begin{equation*}
\Big(p_{\alpha, \varrho}(x)-\frac{\mathbf{c}}{\mathbf{d}} p_{\alpha-1, \varrho}(1)-p_{\alpha, \varrho}(1) \Big) \leq  k \Big( \frac{1}{2^{j+1}}\Big)^{\alpha-1} \Big(\frac{1}{2^j} + \frac{\mathbf{c}}{\mathbf{d}} \Big), ~~~~~\forall x \in [0,1].
\end{equation*}
Now, since the Haar function satisfies the orthogonal property, so we have,
\begin{align*}
a_{2^j +k+1} &\leq \epsilon \Big( \frac{1}{2^{j+1}}\Big),\\
a_{2^s +t+1} &\leq \epsilon \Big( \frac{1}{2^{s+1}}\Big).
\end{align*}
So, \eqref{P3_3_1_1} reduces to,
\begin{align*}
\|E_M\|_2^2 &\leq  k^2 \epsilon^2 \sum_{j=J+1}^{\infty} \sum_{k=0}^{2^j-1} \sum_{s=J+1}^{\infty} \sum_{t=0}^{2^s-1} \Big( \frac{1}{2^{j+1}}\Big) \Big( \frac{1}{2^{s+1}}\Big) \Big( \frac{1}{2^{j+1}}\Big)^{\alpha-1} \Big( \frac{1}{2^j} + \frac{\mathbf{c}}{\mathbf{d}}\Big) \Big( \frac{1}{2^{s+1}}\Big)^{\alpha-1} \Big( \frac{1}{2^s} + \frac{\mathbf{c}}{\mathbf{d}}\Big),\\ 
 &= k^2 \epsilon^2 \Big\{ \Big( \frac{1}{2^{J+2}}\Big)^\alpha \Big(\frac{2^\alpha}{2^\alpha -1} \Big) + \frac{\mathbf{c}}{\mathbf{d}} \Big( \frac{1}{2^{J+1}}\Big)^{\alpha-1} \frac{1}{2 (2^{\alpha -1} -1} \Big\}^2.
 \end{align*}
 Hence we arrive at, 
 \begin{equation}\label{P3_4_1_1}
 \|E_M\|_2 \leq k \epsilon \Big( \frac{1}{2^{J+1}}\Big)^{\alpha-1} \Big\{ \Big( \frac{1}{2^{J+1}}\Big) \Big(\frac{1}{2^\alpha -1} \Big) + \frac{\mathbf{c}}{2\mathbf{d}} \Big(\frac{1}{2^{\alpha -1} -1}\Big) \Big\},
\end{equation}
with $M = 2^J$ and $J$ denoting the highest resolution level, the equation \eqref{P3_4_1_1} ensures the convergence of UFHWCM.
\end{proof} 
\subsection{Stability Analysis}

To ensure the stability of UFHWCM, we must monitor the condition number of the system of algebraic equations. To achieve this, we express the system of algebraic equations derived from the Haar wavelet Collocation method as,
\begin{align}
TV=B,
\end{align}
where $T$ is a square matrix ($2M \times 2M$), $B$ is a constants vector, and $V$ represents wavelet coefficients. The condition number of matrix $T$ measures its sensitivity to input changes. It's defined as $\kappa(T) = ||T|| \cdot ||T^{-1}||$, where $||\cdot||$ denotes the matrix norm, indicating how much the output varies with small changes in the input. A low condition number ($\kappa(T)$) implies better stability, while a high $\kappa(T)$ suggests potential numerical instability. To perform a numerical stability check, we employ the following theorem:
\begin{theorem}
\cite{Randall_2007}Consider a numerical method applied to an ordinary differential equation, resulting in a sequence of matrix equations in the form $TV = B$. We define the method as stable under the following conditions:
\begin{enumerate}
    \item The inverse of the matrix $T$, denoted as $T^{-1}$, exists for all sufficiently large values of $M (M > \mathcal{M}_0)$.
    \item There exists a constant $\mathcal{C}$, independent of $M$, such that the norm of the inverse matrix, $||T^{-1}||$, is bounded by $\mathcal{C}$ for all $M> \mathcal{M}_0$.
\end{enumerate}
\end{theorem}
To verify the above result, we present the $||T^{-1}||_2$ for all the test problems with varying $(\alpha,\beta)$ in Table \ref{norm_T_inverse}. From the Table, we can observe that $||T^{-1}||_2$ is approximately bounded by quantity $0.7071$ in all the scenarios. Additionally, we analyze the condition number of $T$ for all the test problems, as presented in Table \ref{Condition_number}. This comparison reveals that the algebraic equations exhibit favorable conditioning. Consequently, proposed method demonstrate numerical stability.
\begin{table}[h]
  \centering
  \caption{$||T^{-1}||_2$ across various test problems.}
  \resizebox{10cm}{2cm}{
    \begin{tabular}{|c|c|c|c|c|c|c|c|}
    \hline
    \multicolumn{1}{|c|}{Test Case} & $(\alpha, \beta)$ & $J=1$   & $J=2$   & $J=3$   & $J=4$   & $J=5$   & $J=6$ \\\hline
     \multirow{3}[0]{*}{Test Case $1$} & (1.75, 0.75) & 0.854193 & 0.730985 & 0.712662 & 0.708603 & 0.707579 & 0.707267 \\\cline{2-8}
          & (1.85, 0.85) & 0.738998 & 0.71568 & 0.709454 & 0.707768 & 0.707306 & 0.707169 \\\cline{2-8}
          & (1.95, 0.95) & 0.712453 & 0.709969 & 0.707943 & 0.707335 & 0.707168 & 0.707123 \\\hline
    \multirow{3}[0]{*}{Test Case $2$} & (1.75, 0.75) & 0.660499 & 0.695483 & 0.704028 & 0.706289 & 0.706891 & 0.70705 \\\cline{2-8}
          & (1.85, 0.85) & 0.658077 & 0.695427 & 0.704098 & 0.706328 & 0.706906 & 0.707055 \\\cline{2-8}
          & (1.95, 0.95) & 0.655434 & 0.695305 & 0.704136 & 0.706352 & 0.706916 & 0.707059 \\\hline
    \multirow{3}[0]{*}{Test Case $3$} & (1.75, 0.75) & 0.713526 & 0.710291 & 0.708267 & 0.707518 & 0.707245 & 0.707152 \\\cline{2-8}
          & (1.85, 0.85) & 0.707934 & 0.708678 & 0.707701 & 0.707308 & 0.707172 & 0.707127 \\\cline{2-8}
          & (1.95, 0.95) & 0.702376 & 0.707187 & 0.707245 & 0.707159 & 0.707123 & 0.707112 \\\hline
    \end{tabular}}
  \label{norm_T_inverse}
\end{table}

\begin{table}[h]
  \centering
  \caption{$\kappa(T)$ across various test problems.}
  \resizebox{10cm}{2cm}{
    \begin{tabular}{|c|c|c|c|c|c|c|c|}
    \hline
    \multicolumn{1}{|c|}{Test Case} & $(\alpha, \beta)$ & $J=1$   & $J=2$   & $J=3$   & $J=4$   & $J=5$   & $J=6$ \\\hline
    \multirow{3}[0]{*}{Test Case $1$} & (1.75, 0.75) & 3.33923 & 4.0592 & 5.6065 & 7.88737 & 11.1397 & 15.7475 \\\cline{2-8}
          & $(1.85, 0.85)$ & 2.52957 & 3.53015 & 4.97638 & 7.03037 & 9.93869 & 14.0532 \\\cline{2-8}
          & $(1.95, 0.95)$ & 2.29839 & 3.4011 & 4.89078 & 6.9466 & 9.8322 & 13.9052 \\\hline
    \multirow{3}[0]{*}{Test Case $2$} & (1.75, 0.75) & 4.66365 & 6.93675 & 9.92064 & 14.0675 & 19.9066 & 28.1556 \\\cline{2-8}
          & $(1.85, 0.85)$ & 4.90846 & 7.34171 & 10.5068 & 14.8991 & 21.0824 & 29.8175 \\\cline{2-8}
          & $(1.95, 0.95)$ & 5.20958 & 7.85523 & 11.2621 & 15.9788 & 22.6138 & 31.9851 \\\hline
    \multirow{3}[0]{*}{Test Case $3$} & (1.75, 0.75) & 1.88252 & 2.65677 & 3.74745 & 5.29371 & 7.48303 & 10.5808 \\\cline{2-8}
          & $(1.85, 0.85)$ & 2.05556 & 2.92168 & 4.12684 & 5.83089 & 8.24234 & 11.654 \\\cline{2-8}
          & $(1.95, 0.95)$ & 2.28755 & 3.29012 & 4.66487 & 6.59901 & 9.3316 & 13.1956 \\\hline
    \end{tabular}}
  \label{Condition_number}
\end{table}

\section{Numerical Illustration}\label{P3_examples}
\textbf{Residual Error:}
 Residual error is denoted and defined as,
$$ E_{res} = \max_{[0,1]}\left|D^\alpha y_M(x) + \frac{\lambda}{x^\beta} D^\beta y_M(x) + f(y_M)\right|,$$
where $y_M(x)$ is the Haar solution.\\
\textbf{Rate of Convergence:} The rate of convergence is determined by the following formula,
$$RoC = \frac{\log{[E_{res}^{2M}/E_{res}^{4M}]}}{\log{2}},$$
where $E_{res}^{2M}$ and $E_{res}^{4M}$ represent the maximum residual errors obtained with $2M$ and $4M,$ collocation points respectively.
\subsection{Test Case 1}\label{P3Test1}
We consider the fractional Lane-Emden equation,
\begin{equation}\label{P3_4_1}
D^\alpha y(x) + \frac{2}{x^\beta} D^\beta y(x) + y^5(x) =0, ~~~~~~~~ ~~1 < \alpha \leq 2, ~~ 0 < \beta \leq 1,~0< x <1,
\end{equation}
subject to boundary condition, $$y'(0) =0, ~~~ y(1) = \frac{\sqrt{3}}{2}.$$
We present our findings in various tables and figures. Table \ref{tab:table1} details the residual error for $J=6$ across various combinations of $\alpha$ and $\beta$. Further insights are provided in Tables \ref{tab:table2}, \ref{tab:table3}, and \ref{tab:table4}, where we show the residual error for different values of $J$ at $\alpha =1.85, 1.95, 1.99$ and $\beta =0.85, 0.95, 0.99$. For a visual representation, we illustrate the solution of problem \eqref{P3_4_1} in Figure \ref{fig1}, showcasing different $\alpha$ and $\beta$ values at $J=6$. Notably, as $\alpha$ approaches 2 and $\beta$ approaches 1, our numerical solution becomes very similar to the exact solution for $\alpha =2$ and $\beta =1$. These observations are reinforced by the corresponding residual errors depicted in Figures \ref{fig2}, \ref{fig3}, \ref{fig4}, and \ref{fig5}, affirming the accuracy of our approach. Based on the tables and plots, it is evident that the residual error decreases as the value of $J$ increases for the given values of $\alpha =1.99, 1.95, \mbox{and}~ 1.85$ and $\beta =0.99, 0.95, \mbox{and}~ 0.85$. This trend is also observed when $J$ is fixed at $6$ with varying values of $\alpha$ and $\beta.$
\begin{table}[H]
  \begin{center}
    \caption{Analyzing residual error behavior in problem \ref{P3Test1} under varying $\alpha$ and $\beta$ conditions ($J=6$).}
    \label{tab:table1}
  \resizebox{0.999\textwidth}{!}{
    \begin{tabular}{|c|c|c|c|c|c|c|}
    \hline
       $(\alpha, ~ \beta)$ & $(1.55, 0.55)$  & $(1.65, 0.65 )$& $(1.75, 0.75)$ & $(1.85, 0.85)$ & $(1.95 ,0.95)$ & $(1.99, 0.99)$ \\
      \hline
      $E_{res}$ &  0.00879432 & 0.00708953 & 0.00474836 & 0.00270619 & 0.00151107 & 0.00142038 \\
        \hline
    \end{tabular}}
  \end{center}
\end{table}
\begin{table}[H]
\begin{center}
\caption{Comparing residual errors in problem \ref{P3Test1} across various $J$ values $(\alpha =1.85, \beta=0.85)$.}
\label{tab:table2}
\resizebox{0.999\textwidth}{!}{
\begin{tabular}{|c|c|c|c|c|c|c|c|c|c|c|}
\hline
$J$ & 1 & 2 & 3 & 4 & 5 & 6 & 7 & 8 & 9 & 10\\
\hline
$E_{res}$ & $0.0562849$	& $0.0344076$	& $0.01961$	& $0.0103913$	& $0.00534055$	& $0.00270619$	& $0.00136202$	& $0.000683228$	& $0.000342167$	& $0.000171222$\\
\hline
RoC & & $0.710021$	& $0.811138$	& $0.916213$	& $0.960316$	& $0.980725$	&$0.990515$	& $0.995309$	& $0.997666$	& $0.998833$
\\
\hline
\end{tabular}
}
\end{center}
\end{table}

\begin{table}[H]
\begin{center}
\caption{Comparing residual errors in problem \ref{P3Test1} across various $J$ values $(\alpha =1.95, \beta=0.95)$.}
\label{tab:table3}
\resizebox{0.999\textwidth}{!}{
\begin{tabular}{|c|c|c|c|c|c|c|c|c|c|c|}
\hline
$J$ & 1 & 2 & 3 & 4 & 5 & 6 & 7 & 8 & 9 & 10\\
\hline
$E_{res}$ & 0.0381916	&0.0216896	&0.011502	&0.0059178	&0.00300099	&0.00151107	&0.000758186 &0.000379757 &0.000190044 &9.50637E-05
\\
\hline
RoC & & 	0.816252	&0.915119	&0.958752	&0.979622	&0.989868	&0.994947	&0.997475	&0.998743	&0.999367\\
\hline
\end{tabular}
}
\end{center}
\end{table}
\begin{table}[H]
\begin{center}
\caption{Comparing residual errors in problem \ref{P3Test1} across various $J$ values $(\alpha =1.99, \beta=0.99)$.}
\label{tab:table4}
\resizebox{0.999\textwidth}{!}{
\begin{tabular}{|c|c|c|c|c|c|c|c|c|c|c|}
\hline
$J$ & 1 & 2 & 3 & 4 & 5 & 6 & 7 & 8 & 9 & 10\\
\hline
$E_{res}$ &0.0367294	&0.0205913	&0.0108599	&0.00557306	&0.00282262	&0.00142038	&0.000712461	&0.0003568	&0.000178542	&8.93067E-05\\
\hline
RoC & & 	0.8349	&0.923024	&0.962469	&0.981435	&0.990758	&0.995394	&0.997695	&0.998852	&0.999423\\
\hline
\end{tabular}
}
\end{center}
\end{table}
\begin{figure}[H]
\begin{tabular}{p{8.5cm}p{8.5cm}}
\includegraphics[height=8cm]{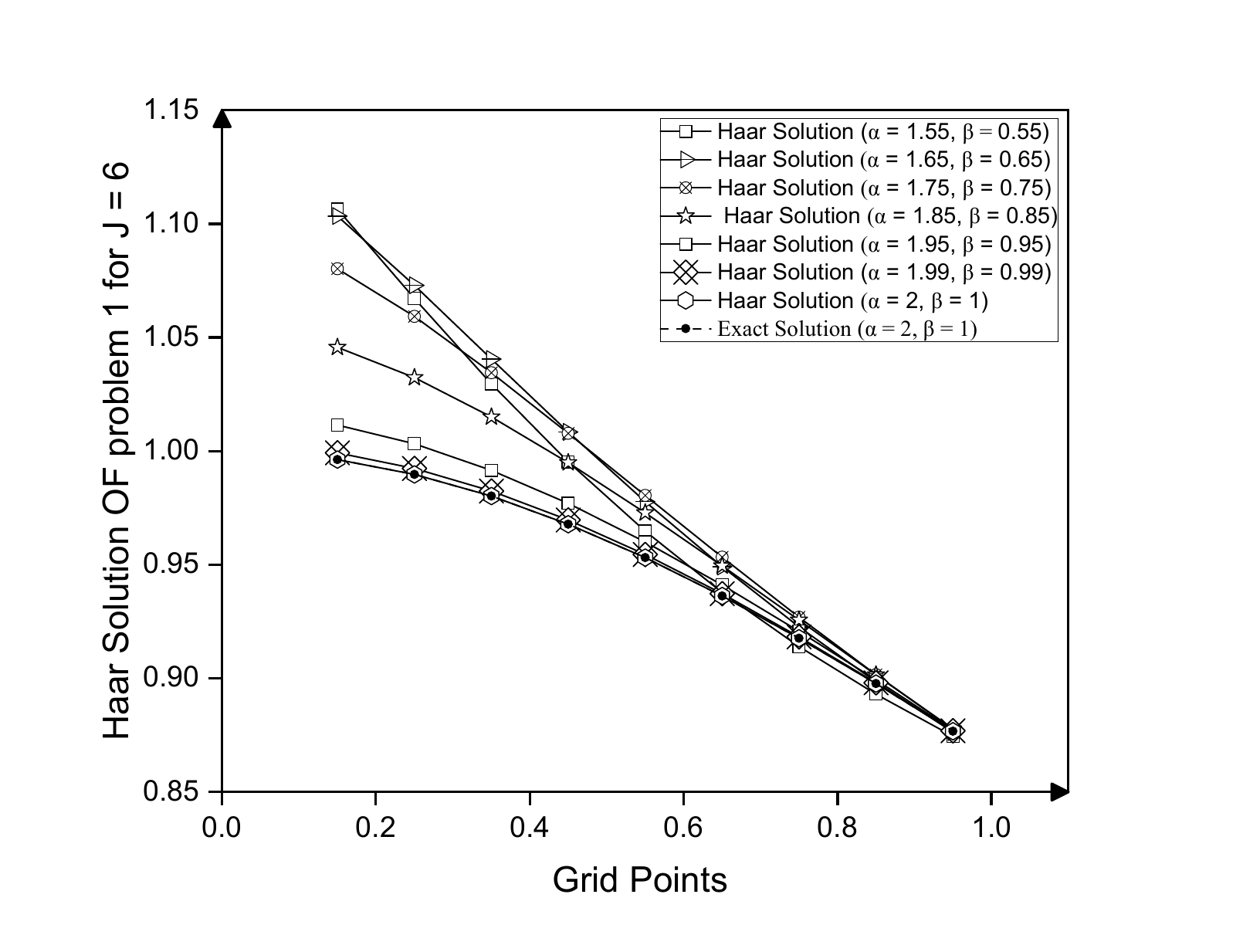}
\caption{Haar solution for problem \ref{P3Test1} at $J=6.$}
\label{fig1}
&
\includegraphics[height=8cm]{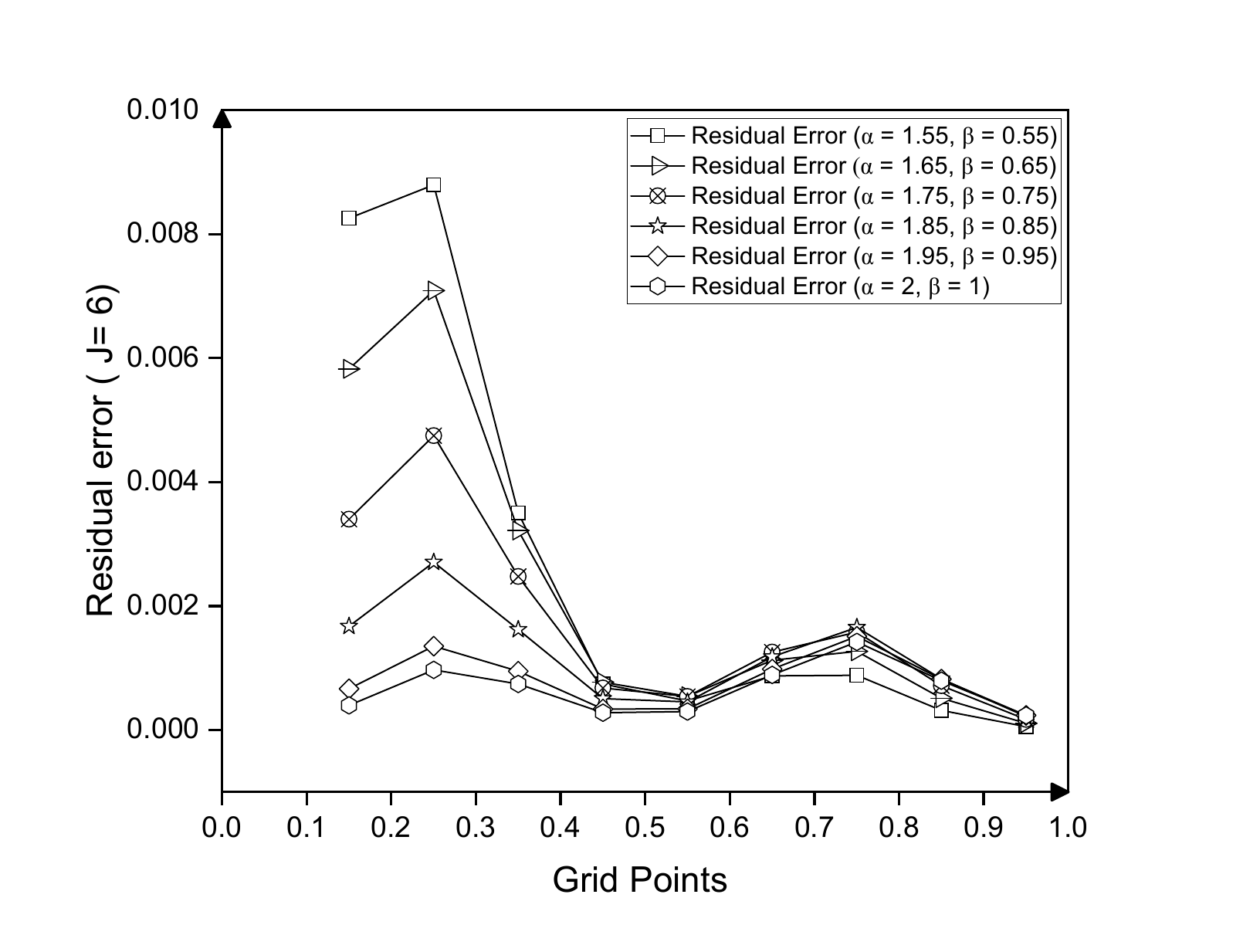}
\caption{Residual error for problem \ref{P3Test1} at $J=6.$}
\label{fig2}
\end{tabular}
\end{figure}
\begin{figure}[H]
\begin{tabular}{p{8.5cm}p{8.5cm}}
\includegraphics[height=8cm]{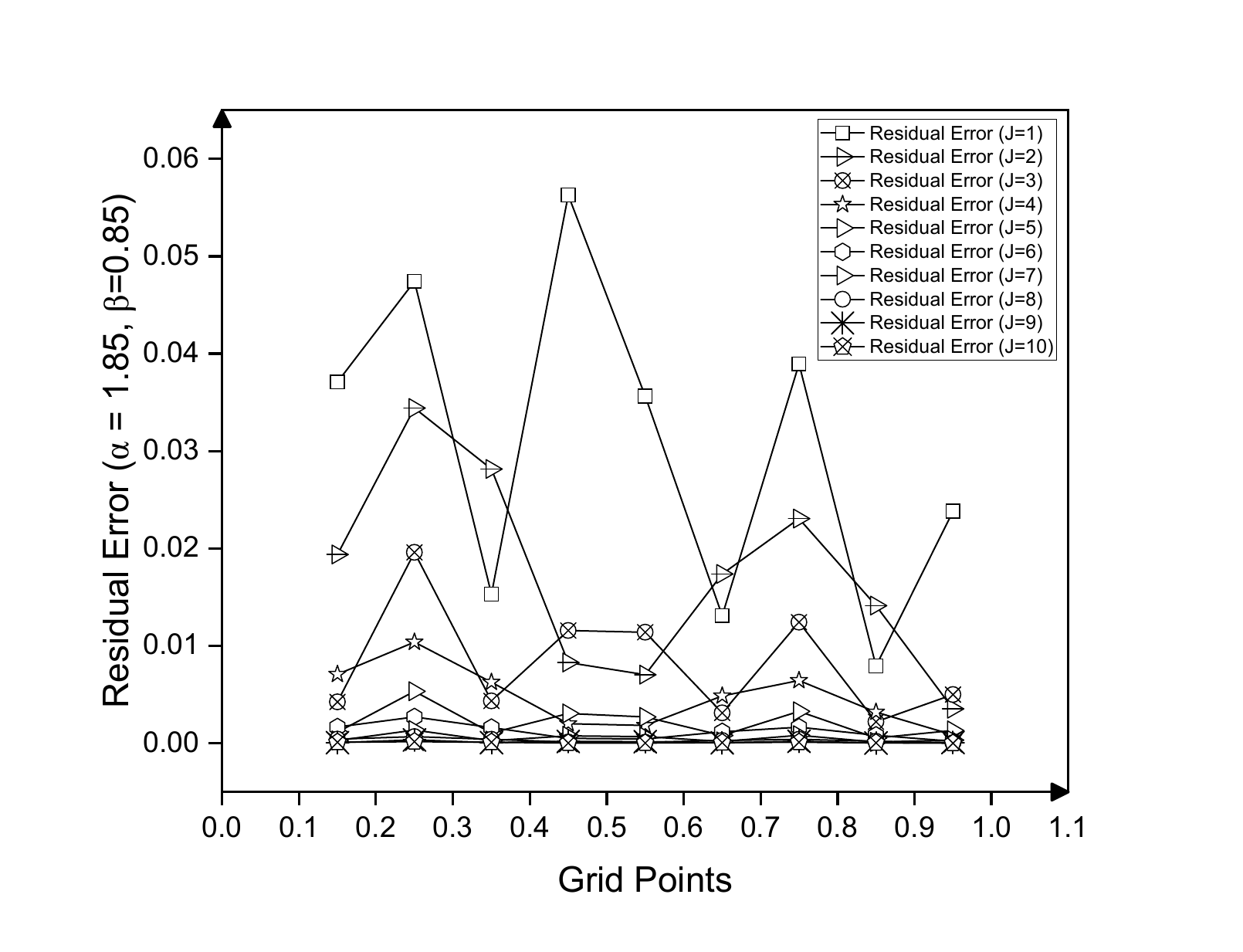}
\caption{Residual error for problem \ref{P3Test1} at $\alpha =1.85, \beta =0.85.$}
\label{fig3}
&
\includegraphics[height=8cm]{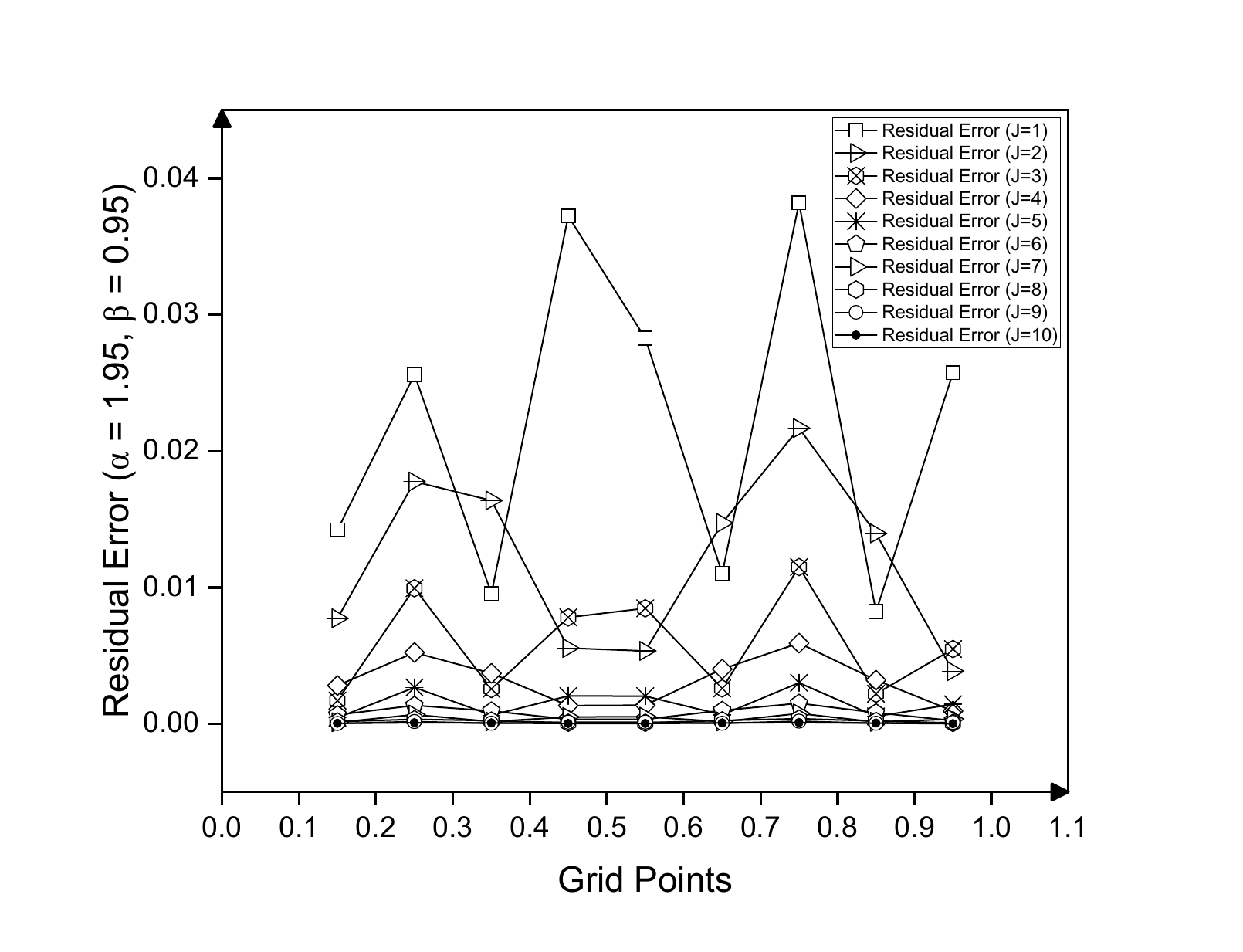}
\caption{Residual error for problem \ref{P3Test1} at $\alpha =1.95, \beta =0.95.$}
\label{fig4}
\end{tabular}
\end{figure}

\begin{figure}[H]
\begin{center}
\includegraphics[height=8cm]{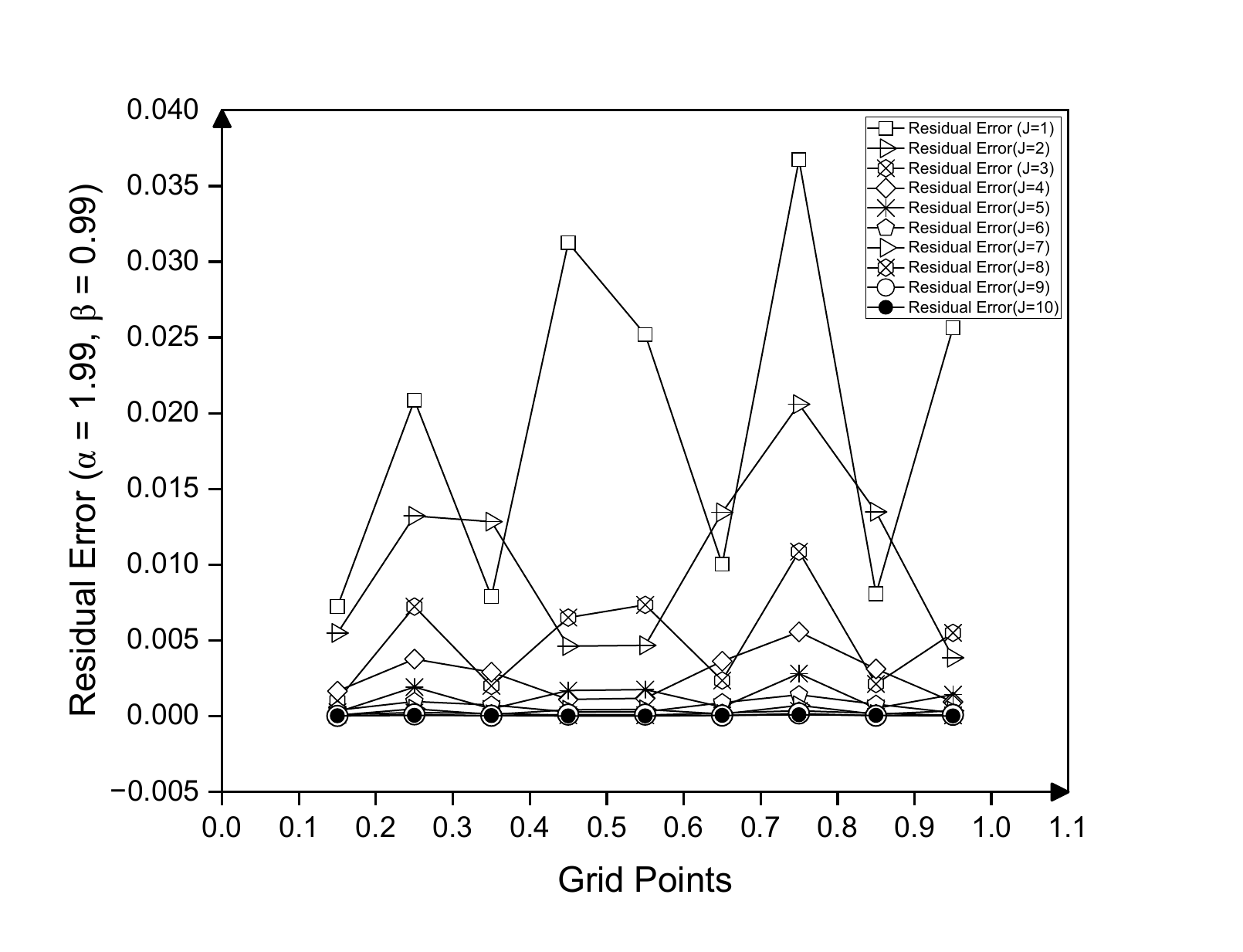}
\caption{Residual error for problem \ref{P3Test1} at $\alpha =1.99, \beta =0.99.$}
\label{fig5}
\end{center}
\end{figure}

\subsection{Test Case 2}\label{P3Test2}
We consider the fractional Lane-Emden equation,
\begin{equation}\label{P3_4_2}
D^\alpha y(x) + \frac{2}{x^\beta} D^\beta y(x) + \exp{(-y(x))} =0, ~~~~~~~~~1 < \alpha \leq 2, ~~ 0 < \beta \leq 1,~0< x <1,
\end{equation}
subject to bounadry condition, $$y'(0) =0, ~~~ 2y(1)+ y'(1) =0.$$
We have included a set of tables and figures to provide a comprehensive view of the residual errors for problem \eqref{P3_4_2}. In table \ref{tab:table5}, we list the residual error for $J=6,$ at different values of $\alpha$ and $\beta.$ In table \ref{tab:table6}, we list the residual error for different value of $J,$ at $\alpha =1.85, \beta =0.85.$ In table \ref{tab:table7}, we list the residual error for different value of $J$, at $\alpha =1.95, \beta =0.95.$ In table \ref{tab:table8}, we list the residual error for different value of $J$ at $\alpha =1.99, \beta =0.99.$ We have also included Figure \ref{fig6}, which illustrates the solution for problem \eqref{P3_4_2} for varying values of $\alpha$ and $\beta$. The corresponding residual errors are shown in Figures \ref{fig7}, \ref{fig8}, \ref{fig9}, and \ref{fig10}. Our goal is to provide a clear and comprehensive understanding of the residual errors for problem \eqref{P3_4_2}, and we believe that these tables and figures will be of great assistance. Based on the tables and plots, it is evident that the residual error decreases as the value of $J$ increases for the given values of $\alpha =1.99, 1.95, \mbox{and}~ 1.85$ and $\beta =0.99, 0.95, \mbox{and}~ 0.85$. This trend is also observed when $J$ is fixed at $6$ with varying values of $\alpha$ and $\beta.$
\begin{table}[H]
  \begin{center}
    \caption{Analyzing residual error behavior in problem \ref{P3Test2} under varying $\alpha$ and $\beta$ conditions ($J=6$).}
    \label{tab:table5}
    \resizebox{0.99\textwidth}{!}{
    \begin{tabular}{|c|c|c|c|c|c|c|c|}
    \hline
         $(\alpha, ~ \beta)$ & $(1.55, 0.55)$  & $(1.65, 0.65 )$& $(1.75, 0.75)$ & $(1.85, 0.85)$ & $(1.95 ,0.95)$ & $(1.99, 0.99)$ & $(2, 1)$\\
      \hline
      $E_{res}$ &0.00121633	&0.000921754	&0.000624467	&0.000343944	&0.000358532	&0.00038494	&0.000391484 \\
        \hline
    \end{tabular}}
  \end{center}
\end{table}

\begin{table}[H]
\begin{center}
\caption{Comparing residual errors in problem \ref{P3Test2} across various $J$ values $(\alpha =1.85, \beta=0.85)$.}
\label{tab:table6}
\resizebox{0.999\textwidth}{!}{
\begin{tabular}{|c|c|c|c|c|c|c|c|c|c|c|}
\hline
$J$ & 1 & 2 & 3 & 4 & 5 & 6 & 7 & 8 & 9 & 10\\
\hline
$E_{res}$ & 0.0110687	&0.00467641	&0.00232679	&0.00149789	&0.000580376	&0.000343944	&0.000145045	&8.86239E-05	&3.62587E-05	&2.20037E-05\\
\hline
RoC & &	1.24301	&1.00706	&0.635409	&1.36787	&0.754814	&1.24567	&0.710733	&1.28937	&0.720581\\
\hline
\end{tabular}
}
\end{center}
\end{table}

\begin{table}[H]
\begin{center}
\caption{Comparing residual errors in problem \ref{P3Test2} across various $J$ values $(\alpha =1.95, \beta=0.95)$.}
\label{tab:table7}
\resizebox{0.999\textwidth}{!}{
\begin{tabular}{|c|c|c|c|c|c|c|c|c|c|c|}
\hline
$J$ & 1 & 2 & 3 & 4 & 5 & 6 & 7 & 8 & 9 & 10\\
\hline
$E_{res}$ & 0.0110906	&0.00563392	&0.00284327	&0.00142867	&0.000716145	&0.000358532	&0.000179382	&8.97201E-05	&4.48673E-05	&2.24355E-05\\
\hline
RoC & &	0.977126	&0.986588	&0.992878	&0.996349	&0.99815	&0.999067	&0.999532	&0.999767	&0.999881\\
\hline
\end{tabular}
}
\end{center}
\end{table}

\begin{table}[H]
\begin{center}
\caption{Comparing residual errors in problem \ref{P3Test2} across various $J$ values $(\alpha =1.99, \beta=0.99)$.}
\label{tab:table8}
\resizebox{0.999\textwidth}{!}{
\begin{tabular}{|c|c|c|c|c|c|c|c|c|c|c|}
\hline
$J$ & 1 & 2 & 3 & 4 & 5 & 6 & 7 & 8 & 9 & 10\\
\hline
$E_{res}$ &0.0116877	&0.00600048	&0.00304198	&0.00153166	&0.000768526	&0.00038494	&0.000192639	&0.000096362	&4.81916E-05	&2.40985E-05\\
\hline
RoC & &	0.961841	&0.980067	&0.989915	&0.99493	&0.99746	&0.998734	&0.999364	&0.999683	&0.999838\\
\hline
\end{tabular}
}
\end{center}
\end{table}

\begin{figure}[H]
\begin{tabular}{p{8.5cm}p{8.5cm}}
\includegraphics[height=8cm]{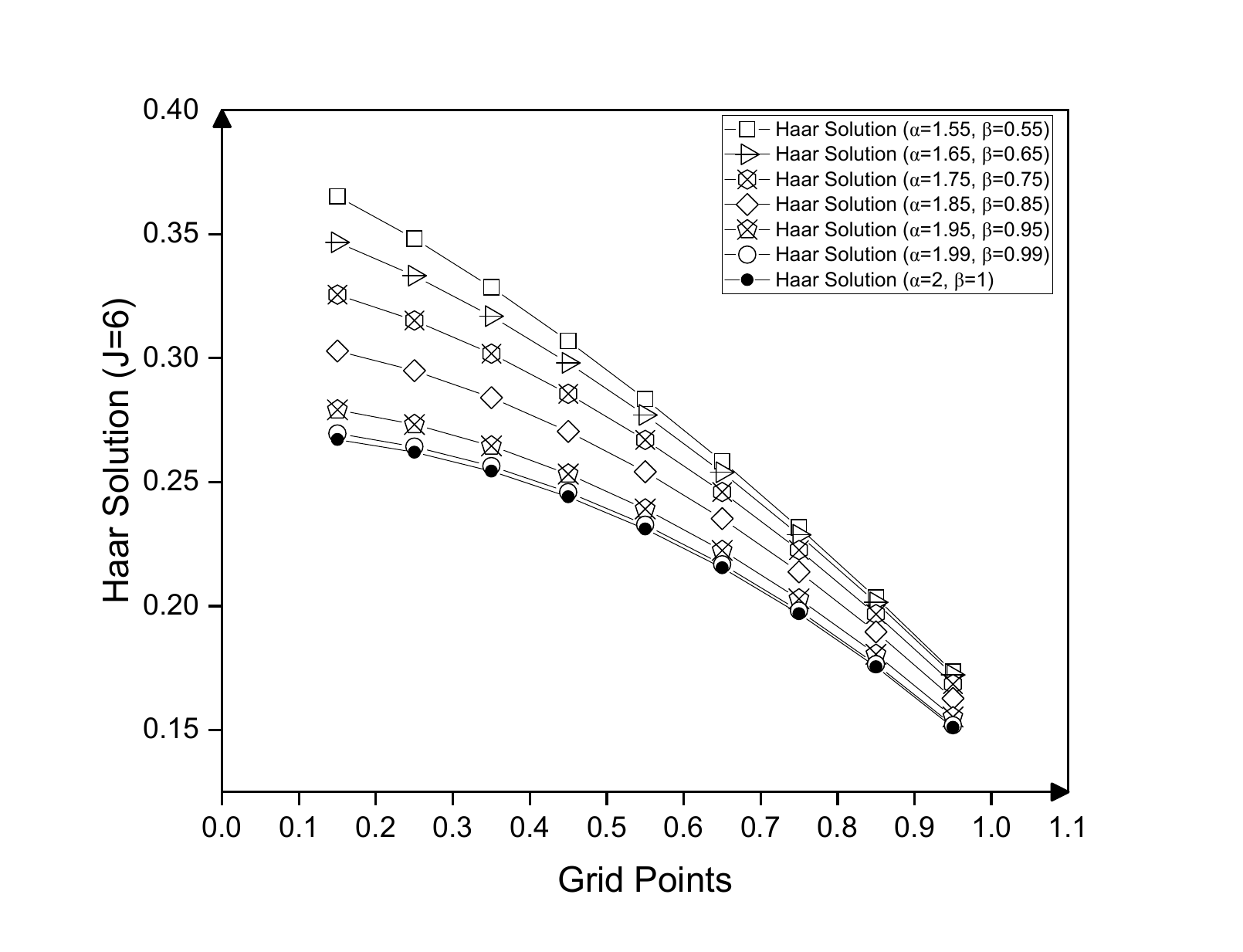}
\caption{Haar solution for problem \ref{P3Test2} at $J=6.$}
\label{fig6}
&
\includegraphics[height=8cm]{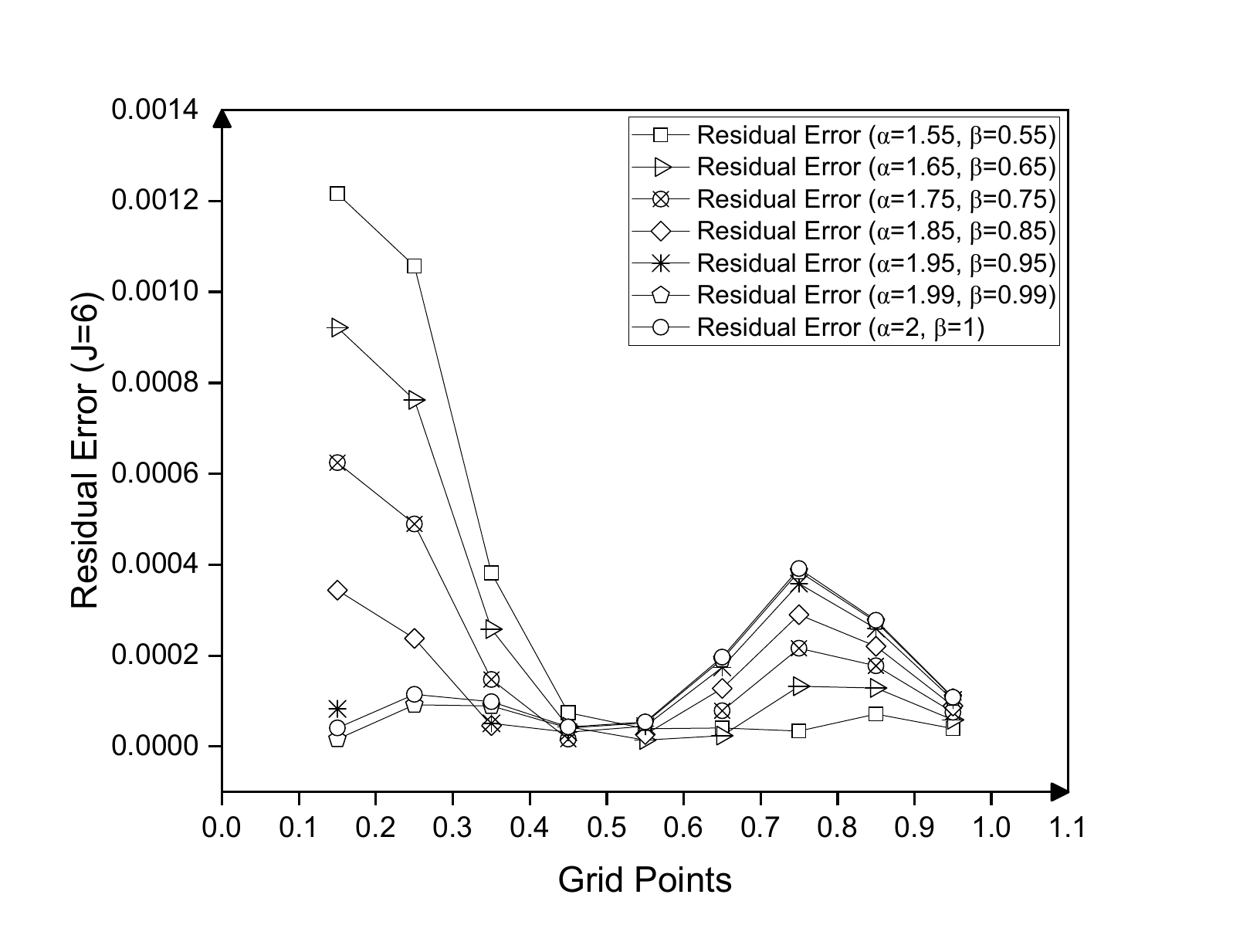}
\caption{Residual error for problem \ref{P3Test2} at $J=6.$}
\label{fig7}
\end{tabular}
\end{figure}
\begin{figure}[H]
\begin{tabular}{p{8.5cm}p{8.5cm}}
\includegraphics[height=8cm]{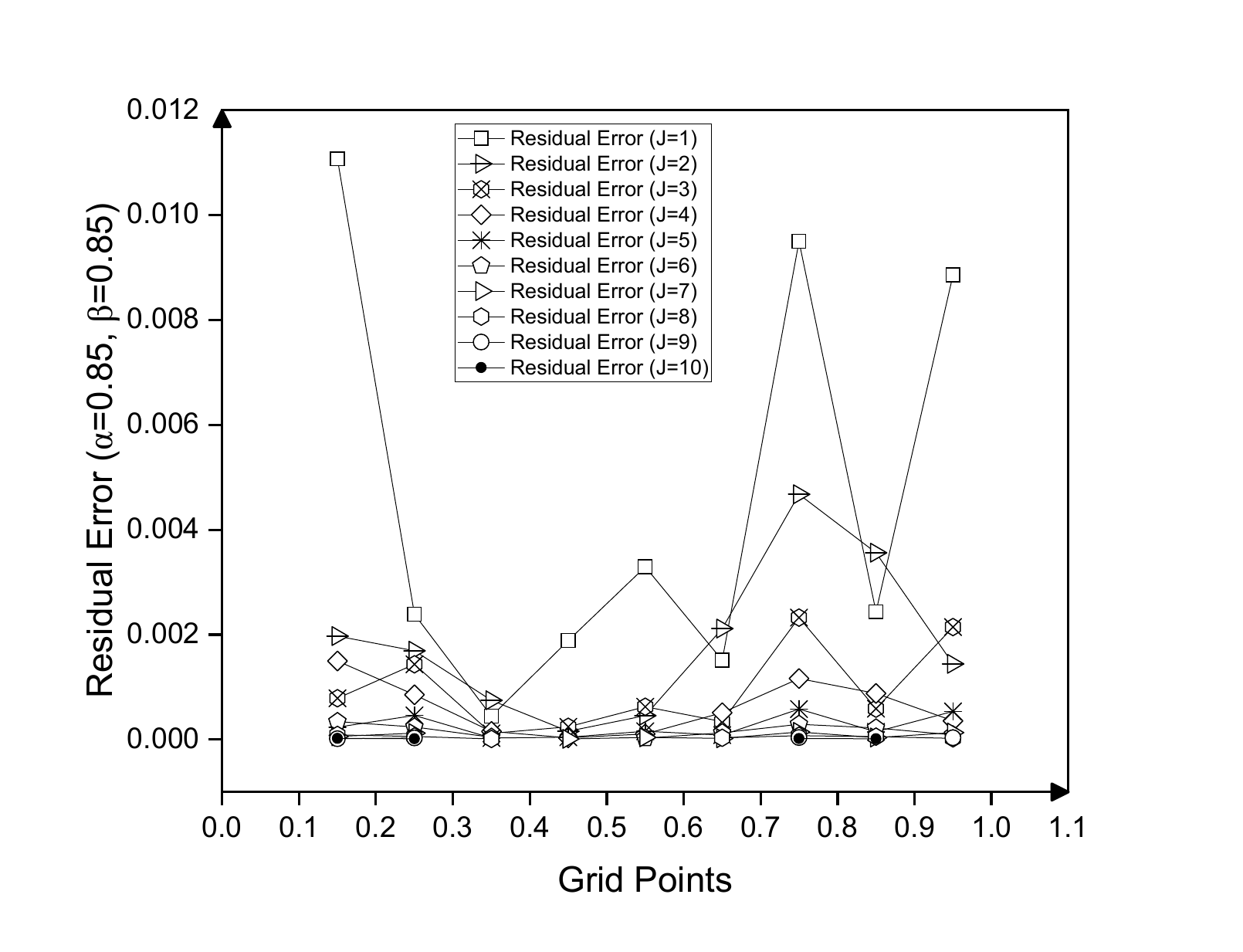}
\caption{Residual error for problem \ref{P3Test2} at $\alpha =1.85, \beta =0.85.$}
\label{fig8}
&
\includegraphics[height=8cm]{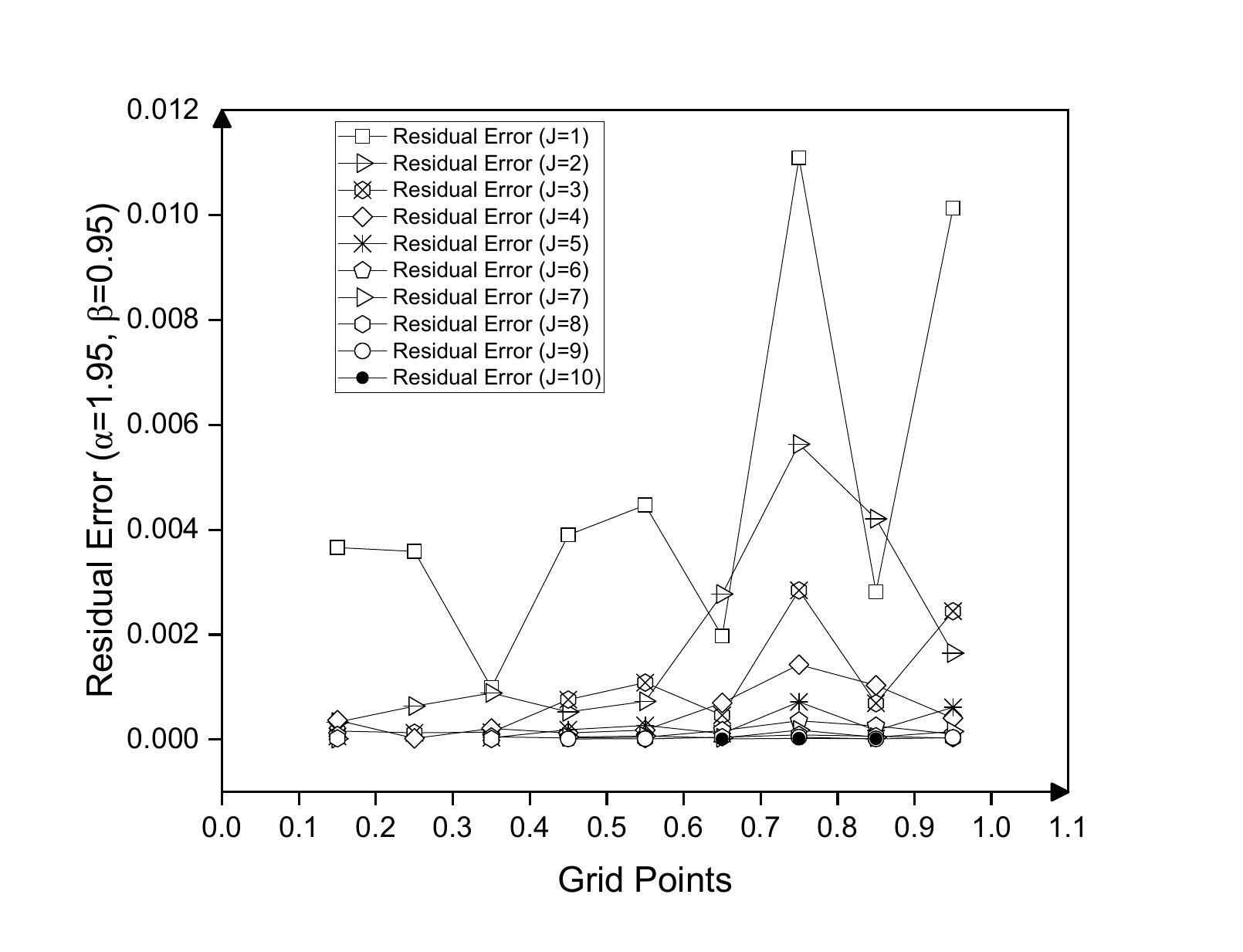}
\caption{Residual error for problem \ref{P3Test2} at $\alpha =1.95, \beta =0.95.$}
\label{fig9}
\end{tabular}
\end{figure}

\begin{figure}[H]
\begin{center}
\includegraphics[height=8cm]{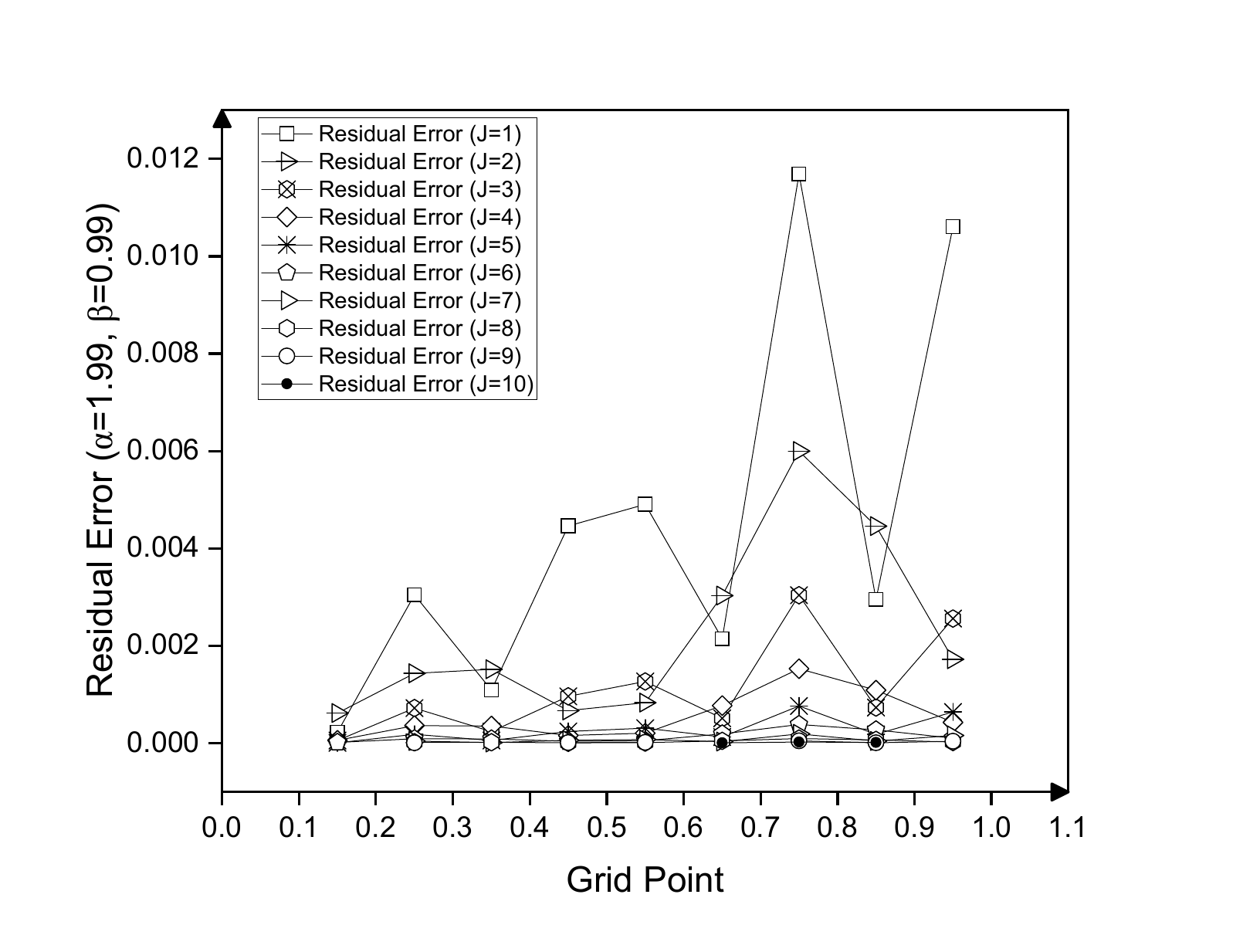}
\caption{Residual error for problem \ref{P3Test2} at $\alpha =1.99, \beta =0.99.$}
\label{fig10}
\end{center}
\end{figure}

\subsection{Test Case 3}\label{P3Test3}
We consider the fractional Lane-Emden equation
\begin{equation}\label{P3_4_3}
D^\alpha y(x) + \frac{1}{x^\beta} D^\beta y(x) + \exp{(y(x))} =0, ~~~~~~~~~1 < \alpha \leq 2, ~~ 0 < \beta \leq 1,~0< x <1,
\end{equation}
subject to boundary condition, $$y'(0) =0, ~~~ y(1) =0.$$
Our results have been presented in a clear and concise manner, demonstrating the accuracy of our scheme. We have compiled all the data in Tables \ref{tab:table9}, \ref{tab:table10}, \ref{tab:table11}, and \ref{tab:table12}, which show the residual error at various values of $J$, $\alpha$, and $\beta$, like Table \ref{tab:table9}, display the residual error at different value of $\alpha$ and $\beta,$ for fixed value of $J=6.$ Table \ref{tab:table10},\ref{tab:table11}, \ref{tab:table12}, display the residual error for different values of $J$ at $\alpha =1.85, 1.95, 1.99$ and $\beta =0.85, 0.95, 0.99.$ These tables provide a strong foundation for our findings. Additionally, we have plotted the solution of the problem in Figure \ref{fig11}, which provides a graphical representation of our results. It is evident that our numerical solution is very close to the exact solution for $\alpha =2$ and $\beta =1$, which clearify by Figures \ref{fig12}, \ref{fig13}, \ref{fig14}, and \ref{fig15}. These results are highly reliable and can be used to support future research in this area.

\begin{table}[H]
  \begin{center}
    \caption{Analyzing residual error behavior in problem \ref{P3Test3} under varying $\alpha$ and $\beta$ conditions ($J=6$).}
    \label{tab:table9}
    \resizebox{0.99\textwidth}{!}{
    \begin{tabular}{|c|c|c|c|c|c|c|}
    \hline
        $(\alpha, ~ \beta)$ & $(1.55, 0.55)$  & $(1.65, 0.65 )$& $(1.75, 0.75)$ & $(1.85, 0.85)$ & $(1.95 ,0.95)$ & $(1.99, 0.99)$ \\
      \hline
      $E_{res}$ & 0.00648987	&0.00479111	&0.00329462	&0.0020617	&0.0016966	&0.00155277\\
        \hline
    \end{tabular}}
  \end{center}
\end{table}

\begin{table}[H]
\begin{center}
\caption{Comparing residual errors in problem \ref{P3Test3} across various $J$ values $(\alpha =1.85, \beta=0.85)$.}
\label{tab:table10}
\resizebox{0.999\textwidth}{!}{
\begin{tabular}{|c|c|c|c|c|c|c|c|c|c|c|}
\hline
$J$ & 1 & 2 & 3 & 4 & 5 & 6 & 7 & 8 & 9 & 10\\
\hline
$E_{res}$ &0.059812	&0.0314603	&0.016125	&0.0081619	&0.00410586	&0.0020617	&0.0010345	&0.000518159	&0.000259307	&0.00012971\\
\hline
RoC & & 0.926903	&0.964233	&0.982322	&0.991221	&0.99385	&0.994901	&0.997467	&0.998734	&0.999371\\
\hline
\end{tabular}
}
\end{center}
\end{table}

\begin{table}[H]
\begin{center}
\caption{Comparing residual errors in problem \ref{P3Test3} across various $J$ values $(\alpha =1.95, \beta=0.95)$.}
\label{tab:table11}
\resizebox{0.999\textwidth}{!}{
\begin{tabular}{|c|c|c|c|c|c|c|c|c|c|c|}
\hline
$J$ & 1 & 2 & 3 & 4 & 5 & 6 & 7 & 8 & 9 & 10\\
\hline
$E_{res}$ & 0.0507101	&0.026287	&0.0133733	&0.00674377	&0.0033861	&0.0016966	&0.000849183	&0.000424812	&0.000212461	&0.000106244\\
\hline
RoC & & 0.947924	&0.974994	&0.987728	&0.993931	&0.996978	&0.998499	&0.999251	&0.999626	&0.999817\\
\hline
\end{tabular}
}
\end{center}
\end{table}
\begin{table}[H]
\begin{center}
\caption{Comparing residual errors in problem \ref{P3Test3} across various $J$ values $(\alpha =1.99, \beta=0.99)$.}
\label{tab:table12}
\resizebox{0.999\textwidth}{!}{
\begin{tabular}{|c|c|c|c|c|c|c|c|c|c|c|}
\hline
$J$ & 1 & 2 & 3 & 4 & 5 & 6 & 7 & 8 & 9 & 10\\
\hline
$E_{res}$ & 0.0470253	&0.0242131	&0.0122762	&0.00617993	&0.00310035	&0.00155277	&0.000777029	&0.000388676	&0.000194378	&9.71992E-05\\
\hline
RoC & & 0.957649	&0.979924	&0.990202	&0.995159	&0.997587	&0.998804	&0.9994	&0.999703	&0.999849\\
\hline
\end{tabular}
}
\end{center}
\end{table}

\begin{figure}[H]
\begin{tabular}{p{8.5cm}p{8.5cm}}
\includegraphics[height=8cm]{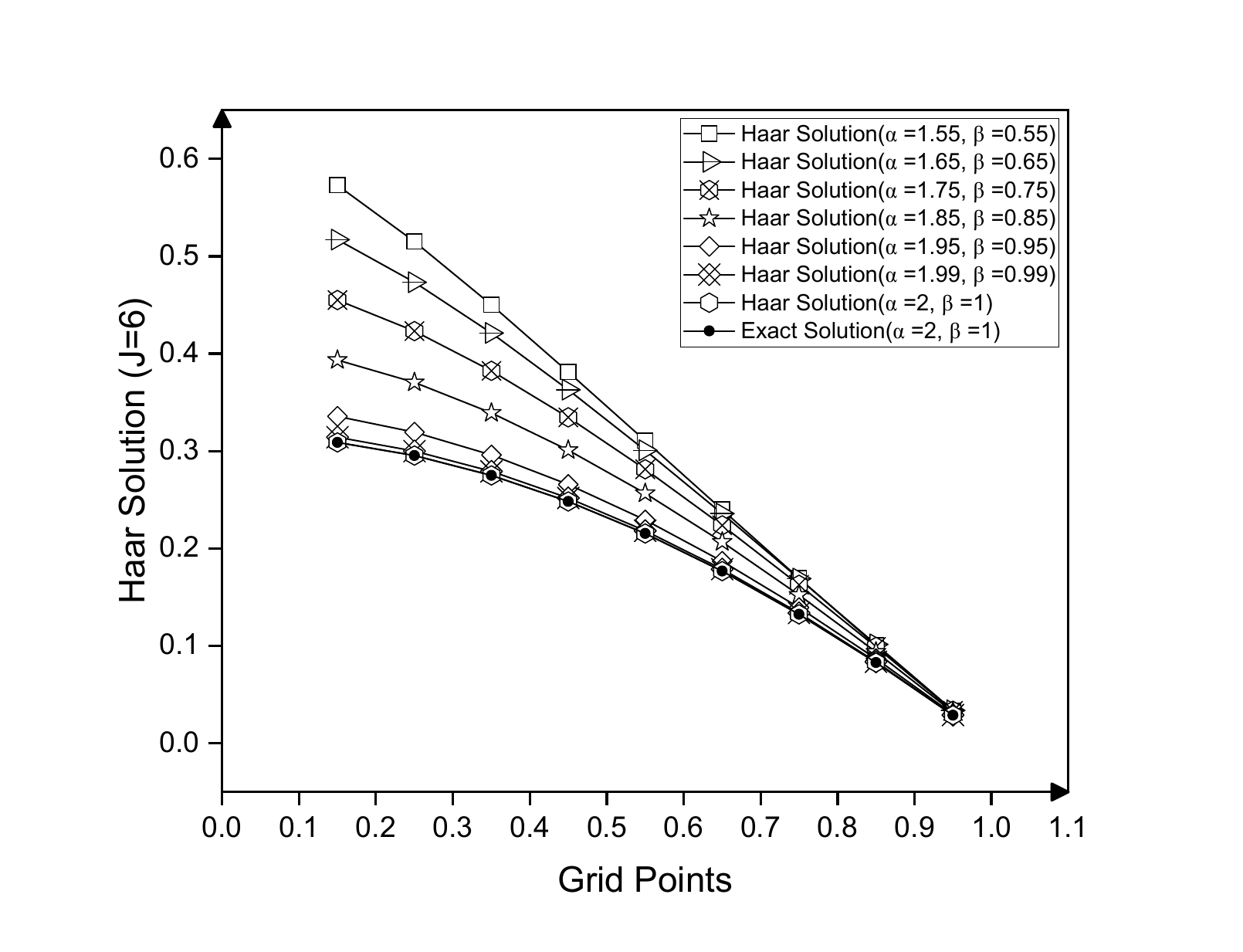}
\caption{Haar solution for problem \ref{P3Test3} for $J=6.$}
\label{fig11}
&
\includegraphics[height=8cm]{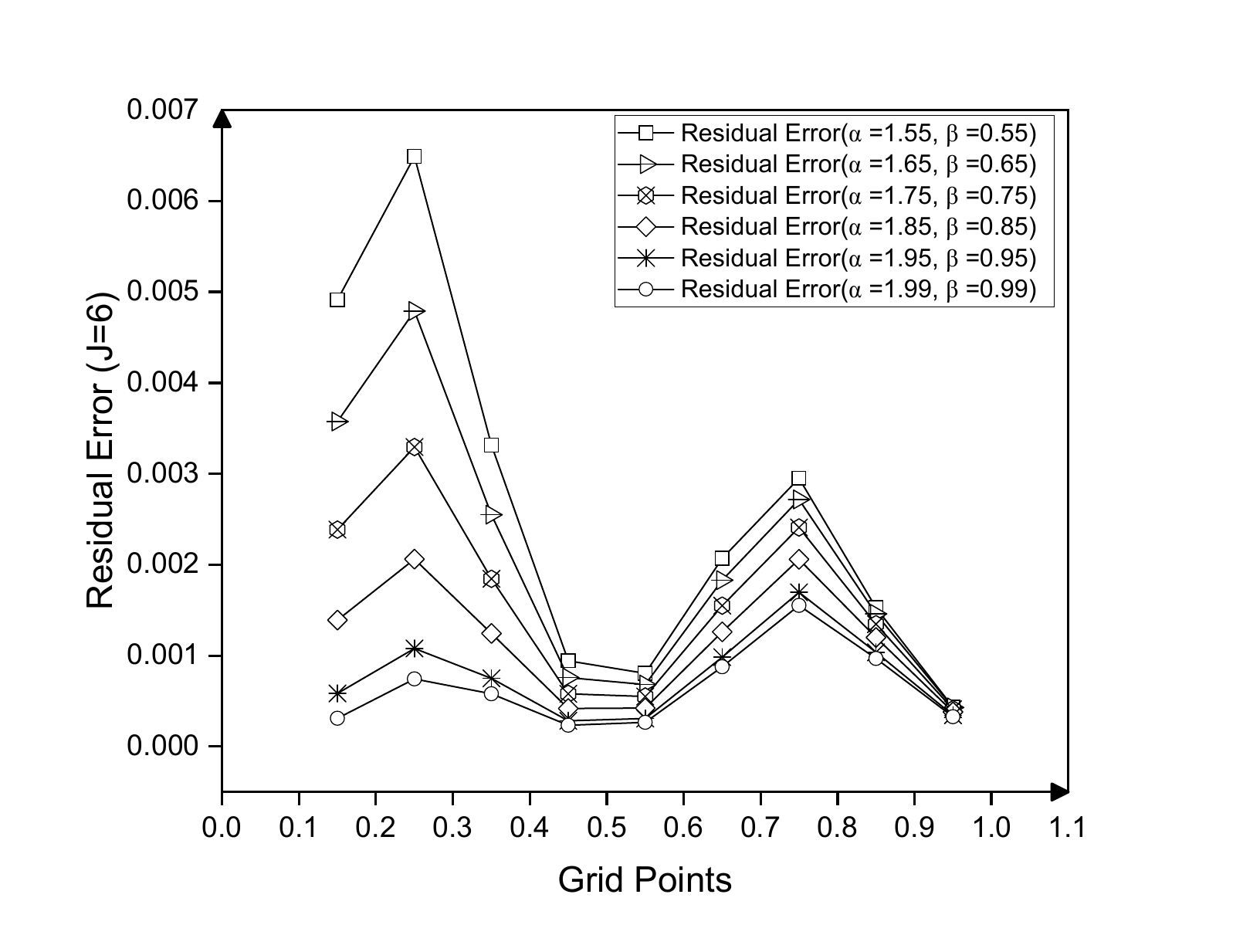}
\caption{Residual error for problem \ref{P3Test3} for $J=6.$}
\label{fig12}
\end{tabular}
\end{figure}
\begin{figure}[H]
\begin{tabular}{p{8.5cm}p{8.5cm}}
\includegraphics[height=8cm]{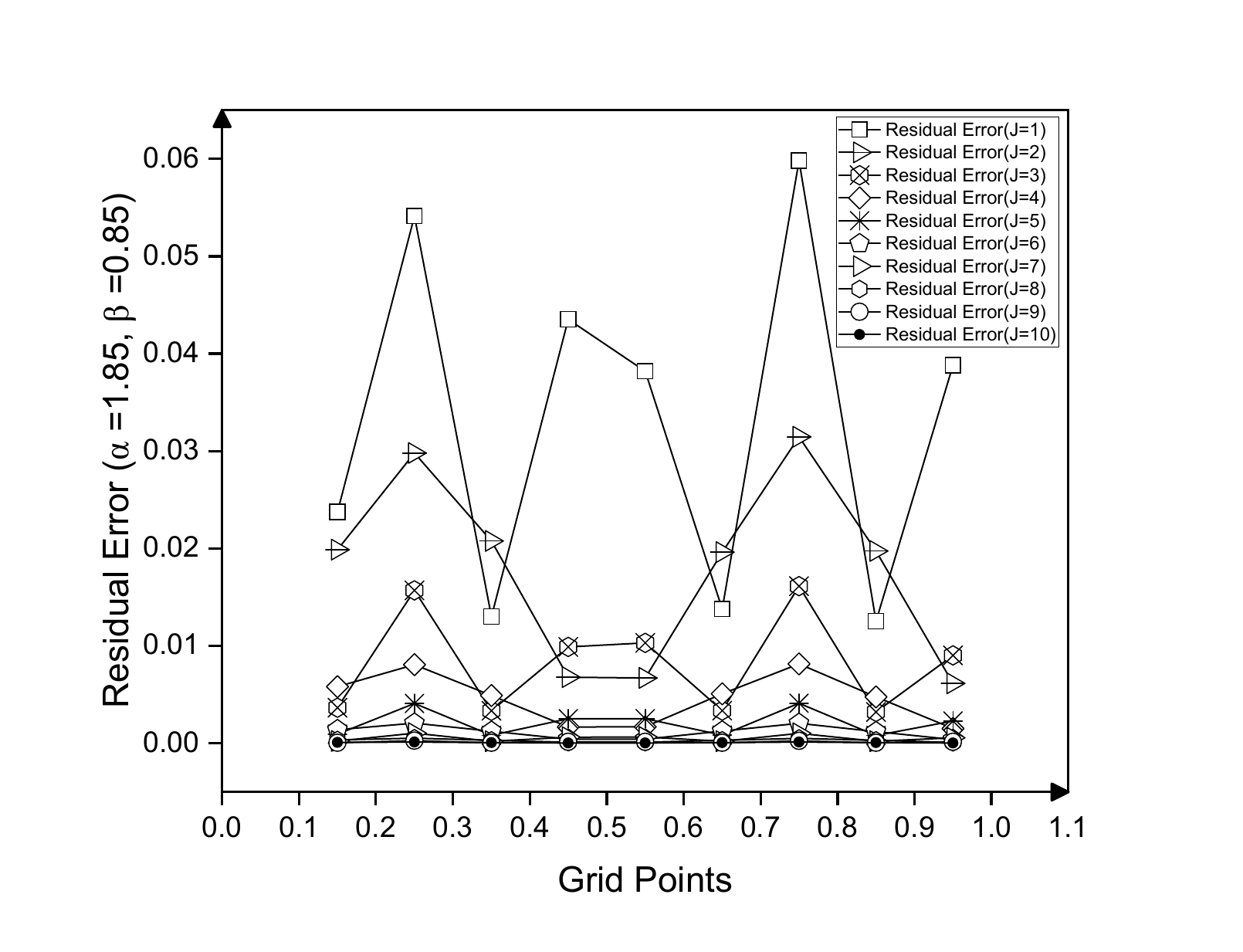}
\caption{Residual error for problem \ref{P3Test3} at $\alpha =1.85, \beta =0.85.$}
\label{fig13}
&
\includegraphics[height=8cm]{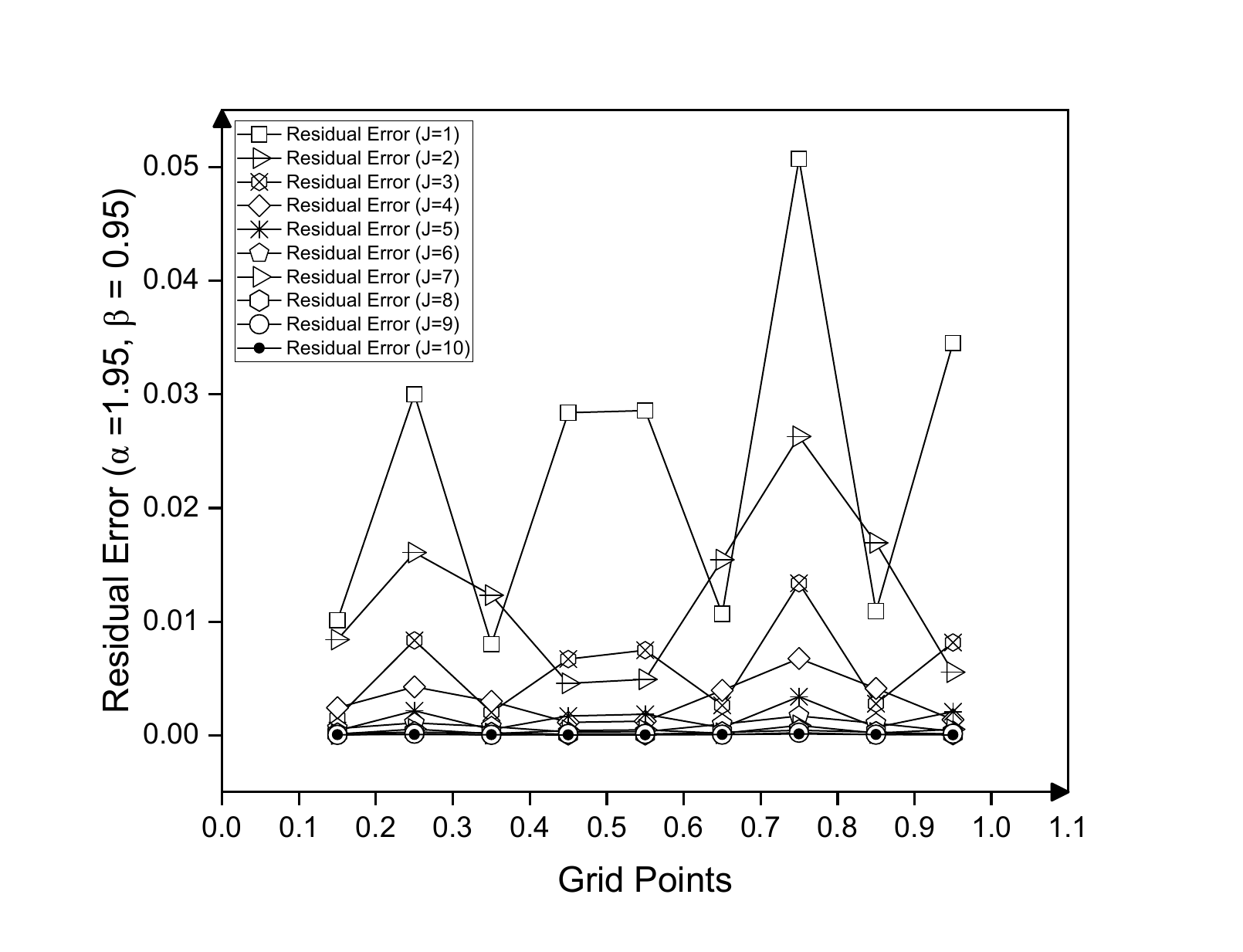}
\caption{Residual error for problem \ref{P3Test3} at $\alpha =1.95, \beta =0.95.$}
\label{fig14}
\end{tabular}
\end{figure}

\begin{figure}[H]
\begin{center}
\includegraphics[height=8cm]{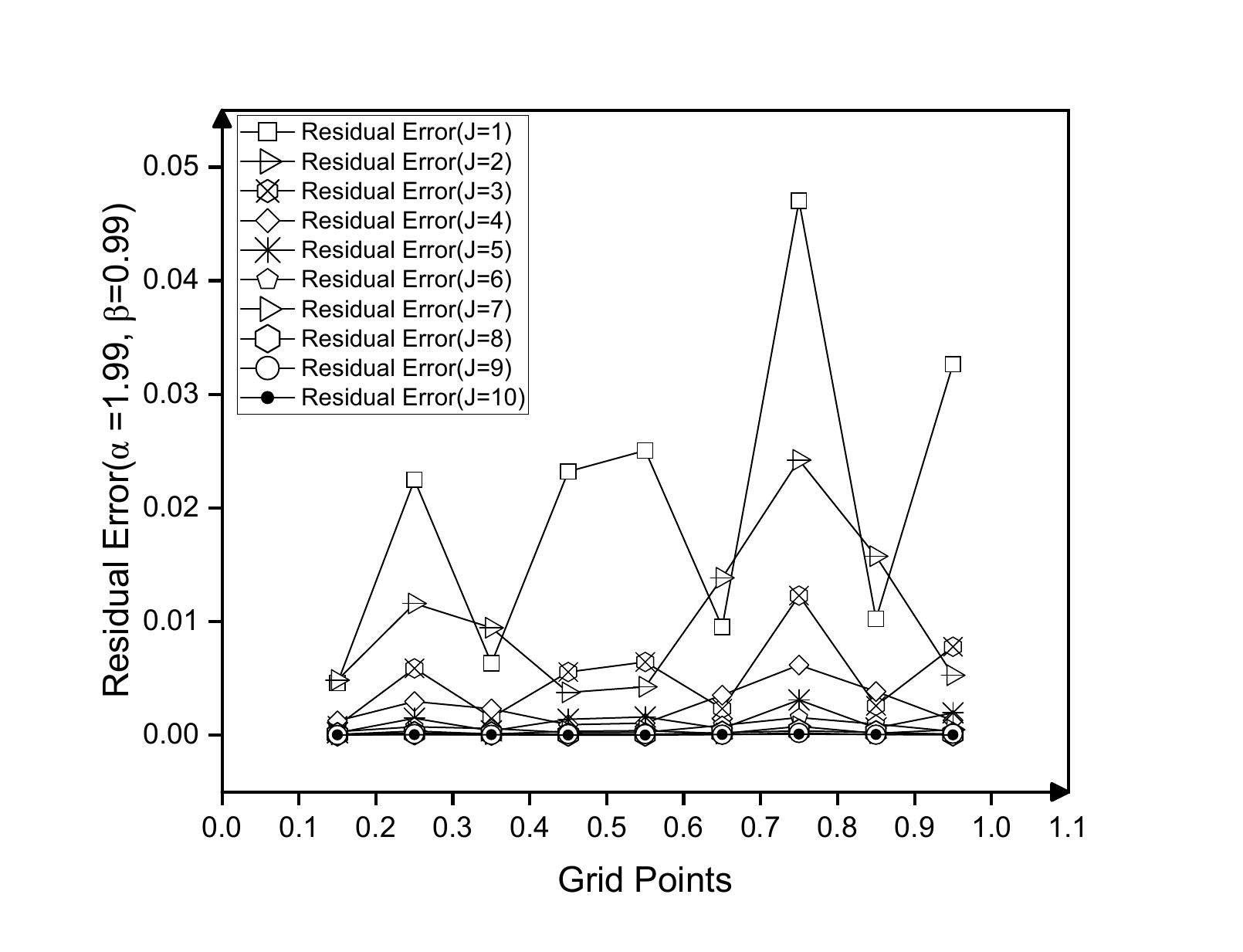}
\caption{Residual error for problem \ref{P3Test3} at $\alpha =1.99, \beta =0.99.$}
\label{fig15}
\end{center}
\end{figure}

\section{Conclusion}\label{P3_conclusion}
We have used the uniform fractional Haar wavelet collocation method to find the solutions for the class of fractional Lane-Emden equations \eqref{P3_1}. We conduct experiments on three test cases \ref{P3Test1}, \ref{P3Test2}, \ref{P3Test3} and conclude that as $(\alpha, \beta)$ approach $(2,1),$ the solutions of the fractional and classical Lane-Emden problems come closer to each other. Our findings highlight the effectiveness of our approach and the accuracy of our results. With clear and concise tables and figures, we have demonstrated the accuracy of our approach and the reliability of our results. Our results showcase the effectiveness of our approach and the accuracy of our findings. Issues related to convergence and stability has also been addressed in the paper. This paper shall lead to new thread of research in near future. 
\section{Acknowledgement}
The 2nd author is very much grateful to all the members of our research group at IIT Patna for their support and help. The work is financially supported by UGC - (December 2019)/2019(NET/CSIR-UGC), NTA Ref. No. 191620007135, New Delhi, India.

\section{Conflict of interest}
Authors don't have any conflict of interest to disclose.

\bibliography{LN3}

\begin{thebibliography}{10}

\bibitem{anderson1981complementary}
Neil Anderson and A.M. Arthurs.
\newblock Complementary extremum principles for a nonlinear model of heat
  conduction in the human head.
\newblock {\em Bulletin of Mathematical Biology}, 43:341--346, 1981.

\bibitem{bellman1965quasilinearization}
R.~E. Bellman and R.~E. Kalaba.
\newblock {\em Quasilinearization and nonlinear boundary-value problems}.
\newblock 1965.

\bibitem{chambre1952solution}
Paul .~L. Chambr{\'e}.
\newblock On the solution of the {P}oisson-{B}oltzmann equation with
  application to the theory of thermal explosions.
\newblock {\em The Journal of Chemical Physics}, 20(11):1795--1797, 1952.

\bibitem{chandrasekhar1939book}
Subrahmanyan Chandrasekhar.
\newblock Book review: An introduction to the study of stellar structure, by s.
  chandrasekhar, 1939.

\bibitem{chandrasekhar1957introduction}
Subrahmanyan Chandrasekhar.
\newblock {\em An introduction to the study of stellar structure}, volume~2.
\newblock Courier Corporation, 1957.

\bibitem{he1998nonlinear}
Ji-Huan He.
\newblock Nonlinear oscillation with fractional derivative and its
  applications.
\newblock In {\em International conference on vibrating engineering},
  volume~98, pages 288--291. Dalian, China, 1998.

\bibitem{he1999some}
Ji-Huan He.
\newblock Some applications of nonlinear fractional differential equations and
  their approximations.
\newblock {\em Bulletin of Science, Technology \& Society}, 15(2):86--90, 1999.

\bibitem{jumarie2006modified}
Guy Jumarie.
\newblock Modified {R}iemann-{L}iouville derivative and fractional taylor
  series of nondifferentiable functions further results.
\newblock {\em Computers \& Mathematics with Applications},
  51(9-10):1367--1376, 2006.

\bibitem{jumarie2007fractional}
Guy Jumarie.
\newblock Fractional partial differential equations and modified
  {R}iemann-{L}iouville derivative new methods for solution.
\newblock {\em Journal of Applied Mathematics and Computing}, 24:31--48, 2007.

\bibitem{jumarie2009table}
Guy Jumarie.
\newblock Table of some basic fractional calculus formulae derived from a
  modified {R}iemann--{L}iouville derivative for non-differentiable functions.
\newblock {\em Applied Mathematics Letters}, 22(3):378--385, 2009.

\bibitem{kilbas2006theory}
Anatoli~Aleksandrovich Kilbas, Hari~M Srivastava, and Juan~J Trujillo.
\newblock Theory and applications of fractional differential equations.
\newblock 204, 2006.

\bibitem{lepik2014haar}
{\"U}lo Lepik and Helle Hein.
\newblock Haar wavelets.
\newblock In {\em Haar wavelets: with applications}, pages 7--20. Springer,
  2014.

\bibitem{Randall_2007}
Randall~J. LeVeque.
\newblock {\em Finite Difference Methods for Ordinary and Partial Differential
  Equations}.
\newblock Society for Industrial and Applied Mathematics, 2007.

\bibitem{majak2015convergence}
J{\"u}ri Majak, BS~Shvartsman, Maarjus Kirs, Meelis Pohlak, and Henrik
  Herranen.
\newblock Convergence theorem for the haar wavelet based discretization method.
\newblock {\em Composite Structures}, 126:227--232, 2015.

\bibitem{mandelzweig2001quasilinearization}
V.~B. Mandelzweig and F.~Tabakin.
\newblock Quasilinearization approach to nonlinear problems in physics with
  application to nonlinear odes.
\newblock {\em Computer Physics Communications}, 141(2):268--281, 2001.

\bibitem{miller1993introduction}
Kenneth~S Miller and Bertram Ross.
\newblock An introduction to the fractional calculus and fractional
  differential equations.
\newblock 1993.

\bibitem{odibat2006approximations}
Zaid Odibat.
\newblock Approximations of fractional integrals and caputo fractional
  derivatives.
\newblock {\em Applied Mathematics and Computation}, 178(2):527--533, 2006.

\bibitem{pandey2010monotone}
RK~Pandey and Amit~K Verma.
\newblock Monotone method for singular bvp in the presence of upper and lower
  solutions.
\newblock {\em Applied mathematics and computation}, 215(11):3860--3867, 2010.

\bibitem{pandey2010solvability}
RK~Pandey and Amit~K Verma.
\newblock On solvability of derivative dependent doubly singular boundary value
  problems.
\newblock {\em Journal of Applied Mathematics and Computing}, 33(1-2):489--511,
  2010.

\bibitem{podlubny1999introduction}
Igor Podlubny.
\newblock An introduction to fractional derivatives, fractional differential
  equations, to methods of their solution and some of their applications.
\newblock {\em Math. Sci. Eng}, 198:340, 1999.

\bibitem{raja2018new}
Muhammad Asif~Zahoor Raja, Muhammad Umar, Zulqurnain Sabir, Junaid~Ali Khan,
  and Dumitru Baleanu.
\newblock A new stochastic computing paradigm for the dynamics of nonlinear
  singular heat conduction model of the human head.
\newblock {\em The European Physical Journal Plus}, 133:1--21, 2018.

\bibitem{saeed2017haar}
Umer Saeed.
\newblock Haar adomian method for the solution of fractional nonlinear
  lane-emden type equations arising in astrophysics.
\newblock {\em Taiwanese Journal of Mathematics}, 21(5):1175--1192, 2017.

\bibitem{scherer2011grunwald}
Rudolf Scherer, Shyam~L Kalla, Yifa Tang, and Jianfei Huang.
\newblock The {G}r{\"u}nwald--{L}etnikov method for fractional differential
  equations.
\newblock {\em Computers \& Mathematics with Applications}, 62(3):902--917,
  2011.

\bibitem{verma2011monotone}
Amit~K Verma.
\newblock The monotone iterative method and zeros of bessel functions for
  nonlinear singular derivative dependent bvp in the presence of upper and
  lower solutions.
\newblock {\em Nonlinear Analysis: Theory, Methods \& Applications},
  74(14):4709--4717, 2011.

\bibitem{math8071045}
Amit~K. Verma, Biswajit Pandit, Lajja Verma, and Ravi~P. Agarwal.
\newblock A review on a class of second order nonlinear singular bvps.
\newblock {\em Mathematics}, 8(7), 2020.

\bibitem{wang2020new}
Kang-Jia Wang.
\newblock A new fractional nonlinear singular heat conduction model for the
  human head considering the effect of febrifuge.
\newblock {\em The European Physical Journal Plus}, 135(11):871, 2020.

\end{thebibliography}
\bibliographystyle{plain}
\end{document}